\documentclass[11pt]{article}

\usepackage[margin=1in]{geometry}

\usepackage[utf8]{inputenc}
\usepackage[T1]{fontenc}
\usepackage[english]{babel}

\usepackage{amsmath,amssymb,amsthm,mathtools}
\usepackage{mathrsfs}
\usepackage{enumitem}
\usepackage{graphicx}
\usepackage[all]{xy}
\usepackage{booktabs}
\usepackage{tabularx}
\usepackage{xcolor}
\usepackage{url}
\usepackage{hyperref}
\usepackage{authblk}

\theoremstyle{plain}
\newtheorem{theorem}{Theorem}
\newtheorem{proposition}[theorem]{Proposition}
\newtheorem{corollary}[theorem]{Corollary}
\newtheorem{lemma}[theorem]{Lemma}

\theoremstyle{definition}
\newtheorem{definition}[theorem]{Definition}

\newtheorem{construction}[theorem]{Construction}

\theoremstyle{remark}
\newtheorem{remark}[theorem]{Remark}

\newcommand{\kk}{\mathbb{K}}
\newcommand{\Rep}{\mathsf{Rep}}
\newcommand{\Ext}{\operatorname{Ext}}
\newcommand{\Hom}{\operatorname{Hom}}
\newcommand{\End}{\operatorname{End}}
\newcommand{\Aut}{\operatorname{Aut}}
\newcommand{\Span}{\operatorname{span}}
\newcommand{\id}{\mathrm{id}}

\newcommand{\Gl}[1]{G_l^{#1}}
\newcommand{\Gr}[1]{G_r^{#1}}

\newcommand{\fnm}[1]{#1}
\newcommand{\sur}[1]{#1}
\newcommand{\jyear}[1]{}
\newcommand{\equalcont}[1]{}
\newcommand{\orgdiv}[1]{#1}
\newcommand{\orgname}[1]{#1}
\newcommand{\orgaddress}[1]{#1}
\newcommand{\street}[1]{#1}
\newcommand{\city}[1]{#1}
\newcommand{\country}[1]{#1}

\title{Cohomological Maschke's Theorem for Generalized Digroups}

\author[1]{\fnm{Jos\'e Gregorio }\sur{Rodr\'iguez-Nieto}\thanks{\texttt{jgrodrig@unal.edu.co}}}

\author[1]{\fnm{Olga Patricia }\sur{Salazar-D\'iaz}\thanks{\texttt{opsalazard@unal.edu.co}}}

\author[2]{\fnm{Andr\'es }\sur{Sarrazola-Alzate}\thanks{\texttt{andres.sarrazola@eia.edu.co}}}

\author[3]{\fnm{Ra\'ul }\sur{Vel\'asquez}\thanks{\texttt{raul.velasquez@udea.co}}}

\affil[1]{\orgdiv{Departamento de Matem\'aticas}, \orgname{Universidad Nacional de Colombia},
	\orgaddress{\street{Carrera 65 No. 59A-110}, \city{Medell\'in}, \country{Colombia}}}

\affil[2]{\orgname{STEAM School, Universidad EIA},
	\orgaddress{\street{Calle 23 AA Sur Nro. 5-200, Kil\'ometro 2+200 Variante al Aeropuerto Jos\'e Mar\'ia C\'ordova},
		\city{Envigado}, \country{Colombia}}}

\affil[3]{\orgdiv{Instituto de Matem\'aticas}, \orgname{Universidad de Antioquia},
	\orgaddress{\street{Calle 67 No. 53-108}, \city{Medell\'in}, \country{Colombia}}}

\date{}

\begin{document}

\maketitle

\begin{abstract}
	We study Maschke-type phenomena in the representation theory of generalized digroups. For a generalized digroup \(D\), we construct an associative enveloping algebra \(A_D\) and prove that \(\Rep(D)\) is equivalent to the category of left \(A_D\)-modules. Under a Maschke-type condition on the group component, we show that short exact sequences split on the \(\rho\)-side, while the obstruction to full splitting is described by cocycles and identified with \(\Ext^1_{\Rep(D)}(Q,W)\). We also derive a spectral sequence with consequences for splitting and non-semisimplicity.
\end{abstract}


	
\section{Introduction}
\label{sec:introduction}

Maschke's theorem (\cite[Chapter~8]{JamesLiebeck2001}, \cite[Section 2]{Serre1977}) is one of the foundational results of finite-dimensional representation theory.
If \(G\) is a finite group and the characteristic of the ground field does not divide the order of \(G\),
then every finite-dimensional \(G\)-module is semisimple. Equivalently, every \(G\)-submodule
admits a \(G\)-stable complement, or, in homological terms, every short exact sequence of
\(G\)-modules splits.

The strength of Maschke's theorem lies not only in its statement but also in its mechanism.
Given a \(G\)-submodule \(W\subseteq V\), one starts with an arbitrary linear projection
\(\pi:V\to W\), which need not be \(G\)-equivariant. Averaging over the group produces the operator
\[
\pi_G=\frac{1}{\lvert G\rvert}\sum_{g\in G} g\cdot \pi \cdot g^{-1},
\]
which is \(G\)-equivariant and still projects onto \(W\). Hence \(\ker(\pi_G)\) is a \(G\)-stable
complement of \(W\). In this way, semisimplicity is obtained by a canonical averaging process on
the group side \cite[Chapter~8, pp.~70--77]{JamesLiebeck2001}.

The aim of this article is to understand how far this philosophy extends to generalized digroups.
Generalized digroups were introduced as a two-sided extension of groups, governed by two
associative products together with a nontrivial set of bar-units, called the halo
\cite{SalazarDiazVelasquezWillsToro2016}. Their structure is richer than that of ordinary groups in
two essential ways: they carry two products, \(\vdash\) and \(\dashv\), linked by mixed
associativity identities, and they may have several bar-units rather than a single one. A
basic structural theorem shows that every generalized digroup is realized in product form
\(G\times E\), where \(G\) is a group acting on the halo \(E\) \cite{MarinGaviria2025}. This
product description is not merely a convenient model; it is the general form of the theory.

The representation theory of generalized digroups reflects this two-layered structure. A
representation consists of two operator families, \(\lambda\) and \(\rho\), encoding the two
products and interacting through compatibility relations. Once one passes to the structural form
\(D\simeq G\times E\), the operators on the \(\rho\)-side depend only on the group component. Thus
the \(\rho\)-side retains a genuinely group-theoretic character and still admits a Maschke-type
averaging argument. The \(\lambda\)-side, by contrast, remains sensitive to the halo and carries
the extra data that distinguishes generalized digroups from groups.

This already suggests the form of the answer. One should not expect a full Maschke theorem in the
category of generalized digroup representations. What survives is a partial splitting phenomenon:
the group-theoretic side still behaves semisimply, while the halo side supports the obstruction to
full splitting. The main purpose of the paper is to isolate that obstruction and describe it
cohomologically \cite[Chapter~III]{Brown1982}\cite[Chapter~2]{Weibel1994}.

Our first main result is categorical. To every generalized digroup \(D\) we associate an
associative enveloping algebra \(A_D\) and prove that the category \(\Rep(D)\) of representations
of \(D\) is equivalent to the category of left \(A_D\)-modules. This places the theory in an
abelian framework with enough injectives and projectives, and allows extension groups to be treated
by standard homological methods \cite[Chapter~2]{Weibel1994}\cite[Chapter~IV]{HiltonStammbach1997}.

Our second main result is a Maschke-type splitting theorem on the \(\rho\)-side. Under the natural
hypothesis that the group component \(G\) is finite and \(\operatorname{char}(\kk)=0\), every short
exact sequence in \(\Rep(D)\) admits a \(\rho\)-equivariant splitting. The obstruction to a full
splitting in \(\Rep(D)\) is then concentrated entirely in the off-diagonal behaviour of the
\(\lambda\)-operators. We show that this obstruction is encoded by an explicit family of cocycles,
which leads to the concrete identification
\[
\Ext^1_{\Rep(D)}(Q,W)\cong Z^1(Q,W)/B^1(Q,W)
\]
in the sense of extension classes \cite{Yoneda1960}\cite[Chapter~2]{Weibel1994}.

Our third main contribution is structural. We construct explicit nonsplit extensions, showing that
\(\Rep(D)\) need not be semisimple even when the group component lies in a classical Maschke
regime. We then derive a spectral sequence separating the group-theoretic and halo-theoretic
contributions to cohomology. Moreover, the forgetful functor to the halo algebra
\(B_E\) preserves injective objects, so the \(E_2\)-page is expressed in terms of
group cohomology with coefficients in \(\Ext_{B_E}^q(Q,W)\). In the Maschke regime,
this spectral sequence collapses and yields criteria for splitting and for
non-semisimplicity.

The guiding theme is therefore clear. Averaging explains why the group side splits. What remains
after averaging is a genuinely halo-theoretic obstruction, invisible in the classical group case
and intrinsic to generalized digroups. The cocycle model for \(\Ext^1\) and the spectral sequence
developed here provide a homological description of that residual obstruction.

The paper is organized as follows. In Section~\ref{sec:preliminaries} we recall the structural
theory of generalized digroups and fix the corresponding notion of representation. In
Section~\ref{sec:enveloping} we define the enveloping algebra \(A_D\) and prove the equivalence
\(\Rep(D)\simeq A_D\text{-}\mathsf{Mod}\). In Section~\ref{sec:cocycle} we establish the
Maschke-type splitting on the \(\rho\)-side and obtain the cocycle description of
\(\Ext^1_{\Rep(D)}(Q,W)\). Section~\ref{sec:examples} provides explicit nonsplit extensions.
Finally, in Section~\ref{sec:spectral} we derive the spectral sequence and its consequences for
splitting and non-semisimplicity.

	\section{Structural preliminaries on generalized digroups}
	\label{sec:preliminaries}
	
	In this section we recall the structural features of generalized digroups that will be used
	throughout the paper. Our aim is twofold. First, we fix the product model \(G\times E\) that
	underlies the classical structure theory and that will later allow us to separate the group side
	from the halo side. Second, we introduce the notion of representation in the form needed for the
	homological developments of the following sections. Although most of the material recalled here is
	structural, it plays a decisive role in the later construction of the enveloping algebra and in the
	cohomological analysis of extension classes.
	
	\subsection{Generalized digroups and their structural form}
	
	We begin by recalling the classical notion of a generalized digroup \cite[Definition 5]{SalazarDiazVelasquezWillsToro2016}.
	
	\begin{definition}\label{def:gd}
		A \emph{generalized digroup} is a quadruple
		\[
		(D,\vdash,\dashv,E),
		\]
		where \(D\) is a nonempty set, \(\vdash,\dashv : D\times D\to D\) are two associative
		binary operations, and \(E\subseteq D\) is a nonempty subset, called the \emph{halo}
		or the set of \emph{bar-units}, such that the following conditions hold:
		\begin{enumerate}[label=\emph{(GD\arabic*)}]
			\item For every \(e\in E\) and every \(x\in D\),
			\[
			x\dashv e=x,
			\qquad
			e\vdash x=x.
			\]
			
			\item For every \(e\in E\) and every \(x\in D\), there exist elements
			\[
			x^{-1}_{l_e},\,x^{-1}_{r_e}\in D
			\]
			such that
			\[
			x^{-1}_{l_e}\dashv x=e,
			\qquad
			x\vdash x^{-1}_{r_e}=e.
			\]
			
			\item For all \(x,y,z\in D\),
			\[
			x\vdash (y\dashv z)=(x\vdash y)\dashv z,
			\]
			and
			\[
			x\dashv (y\dashv z)=x\dashv (y\vdash z),
			\qquad
			(x\vdash y)\vdash z=(x\dashv y)\vdash z.
			\]
		\end{enumerate}
	\end{definition}
	
	\begin{remark}
		The halo \(E\) need not be a singleton. Thus a generalized digroup differs from a group not only
		because it is governed by two products instead of one, but also because it may carry several
		bar-units. This extra geometric direction is one of the basic algebraic sources of the extension
		phenomena studied later.
	\end{remark}
	
	For a fixed \(e\in E\), the left and right inverses with respect to \(e\) organize
	distinguished group parts of the generalized digroup.
	
	\begin{definition}\label{def:left-right-groups}
		Let \(D=(D,\vdash,\dashv,E)\) be a generalized digroup and let \(e\in E\).
		Define
		\[
		\Gl{e}:=\{x^{-1}_{l_e}\mid x\in D\},
		\qquad
		\Gr{e}:=\{x^{-1}_{r_e}\mid x\in D\}.
		\]
		These are called, respectively, the \emph{left group at \(e\)} and the
		\emph{right group at \(e\)}.
	\end{definition}
	
	\begin{remark}
		The notation above is standard in the structure theory of generalized digroups.
		For each fixed bar-unit \(e\), the sets \(\Gl{e}\) and \(\Gr{e}\) carry natural group
		structures induced by \(\dashv\) and \(\vdash\), respectively (\cite[Theorem 2]{SalazarDiazVelasquezWillsToro2016}). In the present article
		the right component will be the one relevant for Maschke-type arguments, whereas the
		structural decomposition is naturally expressed using the left group.
	\end{remark}
	
	A key fact in the classical theory is that generalized digroups admit a product-type
	description with an action on the halo. Since this theorem is structural and will be
	used repeatedly later, we record it explicitly (\cite[Theorem 4 and 5]{SalazarDiazVelasquezWillsToro2016}).
	
	\begin{theorem}[Structure theorem]\label{thm:structure}
		Let \(D=(D,\vdash,\dashv,E)\) be a generalized digroup and let \(e\in E\).
		Then there exists a natural left action
		\[
		\Gl{e}\times E\to E,
		\qquad
		(g,\alpha)\mapsto g\bullet \alpha,
		\]
		such that \(D\) is isomorphic, as a generalized digroup, to
		\[
		\Gl{e}\times E
		\]
		endowed with the operations
		\begin{equation}\label{eq:structural-ops}
			(g,\alpha)\vdash(h,\beta)=(gh,\,g\bullet \beta),
			\qquad
			(g,\alpha)\dashv(h,\beta)=(gh,\,\alpha).
		\end{equation}
		In particular, every generalized digroup is, up to isomorphism, of this form.
	\end{theorem}
	
	\begin{remark}
		Theorem~\ref{thm:structure} shows that product constructions are not merely illustrative examples.
		They constitute the general structural model of the classical theory. Accordingly, later arguments
		may be developed in the product picture without loss of generality.
	\end{remark}
	
	We next isolate the product construction appearing in Theorem~\ref{thm:structure}.
	
	\begin{construction}\label{cons:product-model}
		Let \(G\) be a group with identity \(1\), let \(E\) be a nonempty set, and let
		\[
		G\times E\to E,\qquad (g,\alpha)\mapsto g\bullet \alpha
		\]
		be a left action of \(G\) on \(E\).
		Define two binary operations on \(G\times E\) by
		\begin{equation}\label{eq:product-ops}
			(g,\alpha)\vdash(h,\beta):=(gh,\,g\bullet \beta),
			\qquad
			(g,\alpha)\dashv(h,\beta):=(gh,\,\alpha),
		\end{equation}
		for all \((g,\alpha),(h,\beta)\in G\times E\). Define also
		\[
		\overline{E}:=\{(1,\alpha)\mid \alpha\in E\}\subseteq G\times E.
		\]
	\end{construction}
	
	\begin{proposition}\label{prop:product-gd}
		Let \(G\) be a group acting on a nonempty set \(E\), and endow \(G\times E\) with the
		operations
		\[
		(g,\alpha)\vdash(h,\beta)=(gh,\,g\bullet\beta),
		\qquad
		(g,\alpha)\dashv(h,\beta)=(gh,\,\alpha),
		\]
		together with the halo
		\[
		\overline{E}:=\{(1,\alpha)\mid \alpha\in E\}.
		\]
		Then \((G\times E,\vdash,\dashv,\overline{E})\) is a generalized digroup.
	\end{proposition}
	
	\begin{proof}
		This is standard in the structure theory of generalized digroups. The verification of the
		axioms follows by a direct computation using associativity in \(G\) and the action identity
		\[
		(gh)\bullet\gamma=g\bullet(h\bullet\gamma)
		\qquad
		(g,h\in G,\ \gamma\in E).
		\]
		For completeness, one checks that \(\vdash\) and \(\dashv\) are associative, that every
		\((1,\alpha)\in\overline{E}\) is a bar-unit, and that the left and right inverses relative to
		\((1,\beta)\) are given by
		\[
		(g,\alpha)^{-1}_{l_{(1,\beta)}}=(g^{-1},\beta),
		\qquad
		(g,\alpha)^{-1}_{r_{(1,\beta)}}=(g^{-1},g^{-1}\bullet\beta).
		\]
		The mixed identities are immediate from the definitions of \(\vdash\) and \(\dashv\).
		See, for instance, \cite[Proposition 7]{SalazarDiazVelasquezWillsToro2016}.
	\end{proof}
	
	\begin{remark}\label{rem:right-group-product}
		For every \(\alpha\in E\), the right group at the bar-unit \((1,\alpha)\) is naturally identified
		with \(G\). Indeed, by the proof above,
		\[
		(g,\xi)^{-1}_{r_{(1,\alpha)}}=(g^{-1},\,g^{-1}\bullet\alpha),
		\]
		and therefore
		\[
		\Gr{(1,\alpha)}
		=
		\{(g^{-1},\,g^{-1}\bullet\alpha)\mid g\in G\}.
		\]
		The map
		\[
		G\longrightarrow \Gr{(1,\alpha)},
		\qquad
		g\longmapsto (g^{-1},\,g^{-1}\bullet\alpha),
		\]
		is bijective, and it transports the group structure of \(G\) to the group structure of
		\(\Gr{(1,\alpha)}\) induced by \(\vdash\). This is the group component to which the classical
		Maschke theorem will later be applied.
	\end{remark}
	
\subsection{Representations of generalized digroups}

We now fix the representation-theoretic language used throughout the article \cite[Definition 6]{MarinGaviria2025}.

\begin{definition}\label{def:representation}
	Let \(D=(D,\vdash,\dashv,E)\) be a generalized digroup. A \emph{representation} of \(D\)
	over \(\kk\) is a triple
	\[
	(V,\lambda,\rho),
	\]
	where \(V\) is a \(\kk\)-vector space,
	\[
	\lambda:D\to \End(V),
	\qquad
	\rho:D\to \Aut(V),
	\]
	are maps satisfying, for all \(x,y\in D\) and all \(e\in E\),
	\[
	\begin{alignedat}{3}
		\lambda_{x\dashv y} &= \lambda_x\lambda_y, \qquad &
		\rho_{x\vdash y} &= \rho_x\rho_y, \qquad &
		\rho_e &= \id_V,\\
		\rho_x\lambda_y &= \lambda_{x\vdash y}, \qquad &
		\lambda_x\rho_y &= \lambda_{x\dashv y}. &&
	\end{alignedat}
	\]
\end{definition}

\begin{remark}
	With this convention, the \(\rho\)-side reflects the group component in the usual order.
	In particular, once one passes to the structural form \(D\simeq G\times E\) and proves that
	\(\rho_{(g,\alpha)}\) depends only on \(g\), the resulting map
	\[
	g\longmapsto \rho_g
	\]
	is a genuine representation of the group \(G\).
\end{remark}

\begin{remark}
	The asymmetry between \(\lambda\) and \(\rho\) is structural. The operators \(\rho_x\) are
	required to be invertible, whereas no invertibility is assumed for \(\lambda_x\). This is the
	mechanism that later permits triangular and nilpotent behavior to coexist with semisimplicity
	on the group side.
\end{remark}

\begin{definition}\label{def:morphisms}
	Let \((V,\lambda,\rho)\) and \((W,\lambda',\rho')\) be representations of \(D\).
	A \emph{morphism of representations}
	\[
	f:(V,\lambda,\rho)\longrightarrow (W,\lambda',\rho')
	\]
	is a \(\kk\)-linear map \(f:V\to W\) such that, for every \(x\in D\),
	\[
	f\lambda_x=\lambda'_x f,
	\qquad
	f\rho_x=\rho'_x f.
	\]
\end{definition}

\begin{proposition}\label{prop:rep-category}
	Generalized digroup representations over \(\kk\), together with their morphisms,
	form a category, which we denote by \(\Rep(D)\).
\end{proposition}

\begin{proof}
	Let \((V,\lambda,\rho)\) be a representation of \(D\). The identity map \(\id_V\) is
	clearly a morphism, because
	\[
	\id_V\lambda_x=\lambda_x=\lambda_x\id_V,
	\qquad
	\id_V\rho_x=\rho_x=\rho_x\id_V
	\]
	for every \(x\in D\).
	
	Now let
	\[
	f:(V,\lambda,\rho)\to(W,\lambda',\rho'),
	\qquad
	g:(W,\lambda',\rho')\to(U,\lambda'',\rho'')
	\]
	be morphisms of representations. Then \(g\circ f\) is \(\kk\)-linear and, for every \(x\in D\),
	\begin{align*}
		(g\circ f)\lambda_x
		&=g(f\lambda_x)
		=g(\lambda'_x f)
		=(g\lambda'_x)f
		=(\lambda''_x g)f
		=\lambda''_x(g\circ f),
	\end{align*}
	and similarly
	\begin{align*}
		(g\circ f)\rho_x
		&=g(f\rho_x)
		=g(\rho'_x f)
		=(g\rho'_x)f
		=(\rho''_x g)f
		=\rho''_x(g\circ f).
	\end{align*}
	Thus \(g\circ f\) is again a morphism of representations.
	
	Associativity of composition and the identity laws are inherited from the category of
	\(\kk\)-vector spaces. Therefore \(\Rep(D)\) is a category.
\end{proof}

\begin{remark}
	Section~\ref{sec:preliminaries} fixes the structural and representation-theoretic language
	used in the remainder of the paper. In the next section we replace the category \(\Rep(D)\)
	by a more algebraically tractable model, namely the category of modules over an associative
	enveloping algebra.
\end{remark}
	
	\section{The enveloping algebra and the category of representations}
	\label{sec:enveloping}
	
	The purpose of this section is to place the representation theory of a generalized digroup
	inside a standard homological framework. We do this by attaching to every generalized digroup
	\(D\) an associative algebra \(A_D\) whose module category is equivalent to \(\Rep(D)\).
	
	\subsection{The enveloping algebra}
	
	Let \(D=(D,\vdash,\dashv,E)\) be a generalized digroup. Consider two formal copies of the set \(D\),
	denoted by
	\[
	\{L_x\mid x\in D\},
	\qquad
	\{R_x\mid x\in D\}.
	\]
	Let \(W_D\) be the free \(\kk\)-vector space with basis
	\[
	\mathcal{X}_D:=\{L_x,R_x\mid x\in D\}.
	\]
	We denote by
	\[
	T(W_D)=\bigoplus_{n\ge 0} W_D^{\otimes n}
	\]
	the tensor algebra of \(W_D\). Thus \(T(W_D)\) is the free unital associative \(\kk\)-algebra
	generated by the symbols \(L_x\) and \(R_x\), \(x\in D\).
	
	\begin{definition}\label{def:AD}
		Let \(I_D\) be the two-sided ideal of \(T(W_D)\) generated by the elements
		\begin{align}
			L_{x\dashv y}-L_xL_y, \label{eq:AD1}\\
			R_{x\vdash y}-R_xR_y, \label{eq:AD2}\\
			R_e-1 \qquad (e\in E), \label{eq:AD3}\\
			R_xL_y-L_{x\vdash y}, \label{eq:AD4}\\
			L_xR_y-L_{x\dashv y}, \label{eq:AD5}
		\end{align}
		for all \(x,y\in D\). The \emph{enveloping algebra} of \(D\) is the quotient algebra
		\[
		A_D:=T(W_D)/I_D.
		\]
	\end{definition}
	
	\begin{remark}
		Relations \eqref{eq:AD1}--\eqref{eq:AD5} are the formal shadows of the axioms of a
		representation. The symbols \(L_x\) encode the operators \(\lambda_x\), while the symbols
		\(R_x\) encode the operators \(\rho_x\). Accordingly, the algebra \(A_D\) is designed so that
		a representation of \(D\) is exactly the same thing as a left \(A_D\)-module.
	\end{remark}
	
	\subsection{Universal property}
	
	The construction above is best understood through its universal property.
	
	\begin{proposition}\label{prop:universal}
		Let \(B\) be a unital associative \(\kk\)-algebra. Then to give a unital algebra morphism
		\[
		\Phi:A_D\longrightarrow B
		\]
		is equivalent to giving two families
		\[
		(\ell_x)_{x\in D}\subseteq B,
		\qquad
		(r_x)_{x\in D}\subseteq B
		\]
		such that, for all \(x,y\in D\) and all \(e\in E\),
		\begin{align}
			\ell_{x\dashv y}&=\ell_x\ell_y, \label{eq:UP1}\\
			r_{x\vdash y}&=r_xr_y, \label{eq:UP2}\\
			r_e&=1, \label{eq:UP3}\\
			r_x\ell_y&=\ell_{x\vdash y}, \label{eq:UP4}\\
			\ell_xr_y&=\ell_{x\dashv y}. \label{eq:UP5}
		\end{align}
	\end{proposition}
	
	\begin{proof}
		Assume first that \(\Phi:A_D\to B\) is a unital algebra morphism. Let
		\[
		\pi:T(W_D)\longrightarrow A_D
		\]
		be the quotient map, and define
		\[
		\ell_x:=\Phi(\pi(L_x)),
		\qquad
		r_x:=\Phi(\pi(R_x)).
		\]
		Since every generator of the ideal \(I_D\) maps to zero in \(A_D\), its image under \(\Phi\)
		must be zero in \(B\). Therefore
		\begin{align*}
			0
			&=\Phi\bigl(\pi(L_{x\dashv y}-L_xL_y)\bigr)
			=\ell_{x\dashv y}-\ell_x\ell_y,\\
			0
			&=\Phi\bigl(\pi(R_{x\vdash y}-R_xR_y)\bigr)
			=r_{x\vdash y}-r_xr_y,\\
			0
			&=\Phi\bigl(\pi(R_e-1)\bigr)
			=r_e-1,\\
			0
			&=\Phi\bigl(\pi(R_xL_y-L_{x\vdash y})\bigr)
			=r_x\ell_y-\ell_{x\vdash y},\\
			0
			&=\Phi\bigl(\pi(L_xR_y-L_{x\dashv y})\bigr)
			=\ell_xr_y-\ell_{x\dashv y}.
		\end{align*}
		Hence the families \((\ell_x)\) and \((r_x)\) satisfy \eqref{eq:UP1}--\eqref{eq:UP5}.
		
		Conversely, suppose that families \((\ell_x)\) and \((r_x)\) in \(B\) are given and satisfy
		\eqref{eq:UP1}--\eqref{eq:UP5}. Since \(T(W_D)\) is the tensor algebra of \(W_D\), the assignment
		\[
		L_x\longmapsto \ell_x,
		\qquad
		R_x\longmapsto r_x
		\]
		extends uniquely to a unital algebra morphism
		\[
		\widetilde{\Phi}:T(W_D)\longrightarrow B.
		\]
		We claim that \(I_D\subseteq \ker(\widetilde{\Phi})\). Indeed, each generator of \(I_D\) is sent to
		zero:
		\begin{align*}
			\widetilde{\Phi}(L_{x\dashv y}-L_xL_y)
			&=\ell_{x\dashv y}-\ell_x\ell_y
			=0,\\
			\widetilde{\Phi}(R_{x\vdash y}-R_xR_y)
			&=r_{x\vdash y}-r_xr_y
			=0,\\
			\widetilde{\Phi}(R_e-1)
			&=r_e-1
			=0,\\
			\widetilde{\Phi}(R_xL_y-L_{x\vdash y})
			&=r_x\ell_y-\ell_{x\vdash y}
			=0,\\
			\widetilde{\Phi}(L_xR_y-L_{x\dashv y})
			&=\ell_xr_y-\ell_{x\dashv y}
			=0.
		\end{align*}
		Therefore \(\widetilde{\Phi}\) factors uniquely through the quotient \(A_D\), yielding a unital
		algebra morphism
		\[
		\Phi:A_D\longrightarrow B.
		\]
		
		The two constructions are inverse to each other by inspection. This proves the claim.
	\end{proof}
	
	\subsection{A structural lemma}
	
	Before passing to the equivalence of categories, we isolate a simple lemma which guarantees that
	the operators \(R_x\) act invertibly on every \(A_D\)-module. This point is essential, because in
	Definition~\ref{def:representation} the map \(\rho\) takes values in \(\Aut(V)\), not merely in
	\(\End(V)\).
	
	\begin{lemma}\label{lem:sharp-element}
		Let \(D=(D,\vdash,\dashv,E)\) be a generalized digroup and let \(x\in D\). Then there exists an
		element \(x^\sharp\in D\) such that
		\[
		x\vdash x^\sharp\in E
		\qquad\text{and}\qquad
		x^\sharp\vdash x\in E.
		\]
	\end{lemma}
	
	\begin{proof}
		By Theorem~\ref{thm:structure}, the generalized digroup \(D\) is isomorphic to a structural model
		\[
		G\times E
		\]
		with operations
		\[
		(g,\alpha)\vdash(h,\beta)=(gh,\,g\bullet\beta),
		\qquad
		(g,\alpha)\dashv(h,\beta)=(gh,\,\alpha).
		\]
		Write the image of \(x\) under such an isomorphism as \((g,\alpha)\). Define
		\[
		(g,\alpha)^\sharp := (g^{-1},g^{-1}\bullet\alpha).
		\]
		Then
		\[
		(g,\alpha)\vdash(g^{-1},g^{-1}\bullet\alpha)
		=
		(1,\,g\bullet(g^{-1}\bullet\alpha))
		=
		(1,\alpha)\in \overline{E},
		\]
		and also
		\[
		(g^{-1},g^{-1}\bullet\alpha)\vdash(g,\alpha)
		=
		(1,\,g^{-1}\bullet\alpha)\in \overline{E}.
		\]
		Transporting this element back through the isomorphism yields an element \(x^\sharp\in D\) such
		that
		\[
		x\vdash x^\sharp\in E
		\qquad\text{and}\qquad
		x^\sharp\vdash x\in E.
		\]
		This proves the lemma.
	\end{proof}
	
	\subsection{Equivalence with modules over $A_D$}
	
	We now prove the main result of the section.
	
	\begin{theorem}\label{thm:equivalence}
		There is an equivalence of categories
		\[
		\Rep(D)\simeq A_D\text{-}\mathsf{Mod},
		\]
		where \(A_D\text{-}\mathsf{Mod}\) denotes the category of left \(A_D\)-modules.
	\end{theorem}
	
	\begin{proof}
		We construct functors in both directions and prove that they are quasi-inverse.
		
		\smallskip
		\noindent
		\textbf{Step 1: from representations to modules.}
		Let \((V,\lambda,\rho)\) be a representation of \(D\). Set
		\[
		\ell_x:=\lambda_x\in\End_\kk(V),
		\qquad
		r_x:=\rho_x\in\End_\kk(V)
		\]
		for every \(x\in D\). Since \((V,\lambda,\rho)\) satisfies the axioms of
		Definition~\ref{def:representation}, we have
		\begin{align*}
			\ell_{x\dashv y}
			&=\lambda_{x\dashv y}
			=\lambda_x\lambda_y
			=\ell_x\ell_y,\\
			r_{x\vdash y}
			&=\rho_{x\vdash y}
			=\rho_x\rho_y
			=r_xr_y,\\
			r_e
			&=\rho_e
			=\id_V,\\
			r_x\ell_y
			&=\rho_x\lambda_y
			=\lambda_{x\vdash y}
			=\ell_{x\vdash y},\\
			\ell_xr_y
			&=\lambda_x\rho_y
			=\lambda_{x\dashv y}
			=\ell_{x\dashv y}.
		\end{align*}
		Thus Proposition~\ref{prop:universal} applies with \(B=\End_\kk(V)\), and yields a unique
		unital algebra morphism
		\[
		\Phi_{(\lambda,\rho)}:A_D\longrightarrow \End_\kk(V)
		\]
		such that
		\[
		\Phi_{(\lambda,\rho)}(L_x)=\lambda_x,
		\qquad
		\Phi_{(\lambda,\rho)}(R_x)=\rho_x
		\]
		for all \(x\in D\). In other words, \(V\) becomes a left \(A_D\)-module.
		
		Now let
		\[
		f:(V,\lambda,\rho)\longrightarrow (W,\lambda',\rho')
		\]
		be a morphism in \(\Rep(D)\). Then, for every \(x\in D\),
		\[
		f\lambda_x=\lambda'_x f,
		\qquad
		f\rho_x=\rho'_x f.
		\]
		Since the generators \(L_x\) and \(R_x\) act on \(V\) and \(W\) by \(\lambda_x,\rho_x\) and
		\(\lambda'_x,\rho'_x\), respectively, it follows that \(f\) commutes with the action of every
		generator of \(A_D\), hence with the action of every element of \(A_D\). Therefore \(f\) is
		\(A_D\)-linear.
		
		We have thus defined a functor
		\[
		\mathcal{F}:\Rep(D)\longrightarrow A_D\text{-}\mathsf{Mod}.
		\]
		
		\smallskip
		\noindent
		\textbf{Step 2: from modules to representations.}
		Let \(M\) be a left \(A_D\)-module. For each \(x\in D\), define
		\[
		\lambda_x(m):=L_x\cdot m,
		\qquad
		\rho_x(m):=R_x\cdot m,
		\qquad m\in M.
		\]
		We first verify that \(\rho_x\in\Aut_\kk(M)\) for every \(x\in D\). By Lemma~\ref{lem:sharp-element},
		there exists \(x^\sharp\in D\) such that
		\[
		x\vdash x^\sharp\in E
		\qquad\text{and}\qquad
		x^\sharp\vdash x\in E.
		\]
		Using the defining relations of \(A_D\), we compute
		\begin{align*}
			R_xR_{x^\sharp}
			&=R_{x\vdash x^\sharp}
			=1,
		\end{align*}
		because \(x\vdash x^\sharp\in E\), and similarly
		\begin{align*}
			R_{x^\sharp}R_x
			&=R_{x^\sharp\vdash x}
			=1,
		\end{align*}
		because \(x^\sharp\vdash x\in E\). Hence \(R_x\) is invertible in \(A_D\), with inverse
		\(R_{x^\sharp}\). Therefore the operator \(\rho_x\) on \(M\) is invertible, so indeed
		\[
		\rho_x\in \Aut_\kk(M).
		\]
		
		We next verify the five axioms of a representation.
		
		For \emph{(R1)}, let \(m\in M\). Then
		\begin{align*}
			\lambda_{x\dashv y}(m)
			&=L_{x\dashv y}\cdot m\\
			&=(L_xL_y)\cdot m\\
			&=L_x\cdot(L_y\cdot m)\\
			&=\lambda_x(\lambda_y(m)).
		\end{align*}
		Thus
		\[
		\lambda_{x\dashv y}=\lambda_x\lambda_y.
		\]
		
		For \emph{(R2)},
		\begin{align*}
			\rho_{x\vdash y}(m)
			&=R_{x\vdash y}\cdot m\\
			&=(R_xR_y)\cdot m\\
			&=R_x\cdot(R_y\cdot m)\\
			&=\rho_x(\rho_y(m)),
		\end{align*}
		so
		\[
		\rho_{x\vdash y}=\rho_x\rho_y.
		\]
		
		For \emph{(R3)}, if \(e\in E\), then
		\[
		\rho_e(m)=R_e\cdot m=1\cdot m=m.
		\]
		Hence \(\rho_e=\id_M\).
		
		For \emph{(R4)},
		\begin{align*}
			(\rho_x\lambda_y)(m)
			&=\rho_x(\lambda_y(m))\\
			&=R_x\cdot(L_y\cdot m)\\
			&=(R_xL_y)\cdot m\\
			&=L_{x\vdash y}\cdot m\\
			&=\lambda_{x\vdash y}(m).
		\end{align*}
		Therefore
		\[
		\rho_x\lambda_y=\lambda_{x\vdash y}.
		\]
		
		For \emph{(R5)},
		\begin{align*}
			(\lambda_x\rho_y)(m)
			&=\lambda_x(\rho_y(m))\\
			&=L_x\cdot(R_y\cdot m)\\
			&=(L_xR_y)\cdot m\\
			&=L_{x\dashv y}\cdot m\\
			&=\lambda_{x\dashv y}(m).
		\end{align*}
		Hence
		\[
		\lambda_x\rho_y=\lambda_{x\dashv y}.
		\]
		
		Thus \((M,\lambda,\rho)\) is a representation of \(D\).
		
		Now let \(g:M\to N\) be an \(A_D\)-linear map. For every \(x\in D\) and every \(m\in M\),
		\[
		g(\lambda_x(m))
		=
		g(L_x\cdot m)
		=
		L_x\cdot g(m)
		=
		\lambda'_x(g(m)),
		\]
		and similarly
		\[
		g(\rho_x(m))
		=
		g(R_x\cdot m)
		=
		R_x\cdot g(m)
		=
		\rho'_x(g(m)).
		\]
		Hence \(g\) is a morphism of representations. We have therefore defined a functor
		\[
		\mathcal{G}:A_D\text{-}\mathsf{Mod}\longrightarrow \Rep(D).
		\]
		
		\smallskip
		\noindent
		\textbf{Step 3: verification that the functors are quasi-inverse.}
		Let \((V,\lambda,\rho)\) be a representation. Applying \(\mathcal{F}\) endows \(V\) with the
		\(A_D\)-module structure for which
		\[
		L_x\cdot v=\lambda_x(v),
		\qquad
		R_x\cdot v=\rho_x(v).
		\]
		Applying \(\mathcal{G}\) to this module recovers exactly the original operators \(\lambda_x\)
		and \(\rho_x\). Hence
		\[
		\mathcal{G}\circ\mathcal{F}=\id_{\Rep(D)}.
		\]
		
		Conversely, let \(M\) be a left \(A_D\)-module. Applying \(\mathcal{G}\) defines operators
		\(\lambda_x\) and \(\rho_x\) by the actions of \(L_x\) and \(R_x\). Applying \(\mathcal{F}\)
		to the resulting representation recovers the original module structure, because \(A_D\) is
		generated by the classes of the symbols \(L_x\) and \(R_x\). Therefore
		\[
		\mathcal{F}\circ\mathcal{G}=\id_{A_D\text{-}\mathsf{Mod}}.
		\]
		
		The two functors are thus quasi-inverse equivalences of categories.
	\end{proof}
	
	\subsection{Homological consequences}
	
	The categorical equivalence above has immediate formal consequences which will be used throughout
	the paper.
	
	\begin{corollary}\label{cor:abelian}
		The category \(\Rep(D)\) is abelian.
	\end{corollary}
	
	\begin{proof}
		By Theorem~\ref{thm:equivalence}, the category \(\Rep(D)\) is equivalent to the category of left
		modules over the associative algebra \(A_D\). The latter is an abelian category. Since abelianness
		is invariant under equivalence of categories, \(\Rep(D)\) is abelian.
	\end{proof}
	
	\begin{corollary}\label{cor:projectives}
		The category \(\Rep(D)\) has enough projectives.
	\end{corollary}
	
	\begin{proof}
		The category of left modules over any associative algebra has enough projectives. By
		Theorem~\ref{thm:equivalence}, \(\Rep(D)\) is equivalent to such a module category. Hence
		\(\Rep(D)\) also has enough projectives.
	\end{proof}
	
	\begin{corollary}\label{cor:injectives}
		The category \(\Rep(D)\) has enough injectives.
	\end{corollary}
	
	\begin{proof}
		By Theorem~\ref{thm:equivalence}, the category \(\Rep(D)\) is equivalent to the category of left
		modules over the associative algebra \(A_D\). The latter has enough injectives. Hence
		\(\Rep(D)\) also has enough injectives.
	\end{proof}
	
	\begin{corollary}\label{cor:Ext-defined}
		For all objects \(Q,W\in\Rep(D)\) and every integer \(n\ge 0\), the extension group
		\[
		\Ext^n_{\Rep(D)}(Q,W)
		\]
		is well defined.
	\end{corollary}
	
	\begin{proof}
		By Corollary~\ref{cor:abelian}, the category \(\Rep(D)\) is abelian, and by
		Corollary~\ref{cor:injectives}, it has enough injectives. Therefore the right derived functors of
		\[
		\Hom_{\Rep(D)}(Q,-)
		\]
		exist. By definition, these derived functors are precisely the groups
		\[
		\Ext^n_{\Rep(D)}(Q,W).
		\]
	\end{proof}
	
	\begin{remark}
		Section~\ref{sec:enveloping} provides the formal bridge from generalized digroups to homological
		algebra. From this point on, extension classes in \(\Rep(D)\) may be studied either categorically
		or through the algebra \(A_D\). This is the framework in which the Maschke splitting and the
		cocycle model for \(\Ext^1\) will be developed in the next section.
	\end{remark}
	
	\section{\texorpdfstring{Maschke splitting and a cocycle model for $\Ext^1$}{Maschke splitting and a cocycle model for Ext1}}
	\label{sec:cocycle}
	
	We now turn to the cohomological core of the paper. The main point is that, under a Maschke-type
	hypothesis on the group component, every short exact sequence splits on the \(\rho\)-side. The
	remaining obstruction to splitting in the full representation category is then forced to live in
	the \(\lambda\)-side, where it can be encoded by an explicit family of cocycles. The purpose of
	this section is to make that obstruction precise and to identify it with \(\Ext^1\).
	
	\subsection{The group part of the \texorpdfstring{$\rho$}{rho}-action}
	
	The first point is that the \(\rho\)-operators depend only on the group component.
	
	\begin{proposition}\label{prop:rho-depends-only-on-g}
		Let \((V,\lambda,\rho)\) be a representation of \(D=G\times E\). Then, for every \(g\in G\) and
		every \(\alpha,\beta\in E\),
		\[
		\rho_{(g,\alpha)}=\rho_{(g,\beta)}.
		\]
		Consequently, one may define
		\[
		\rho_g:=\rho_{(g,\alpha)}
		\qquad(g\in G),
		\]
		independently of the choice of \(\alpha\in E\). Moreover,
		\[
		\rho_{gh}=\rho_g\rho_h
		\qquad(g,h\in G),
		\]
		so that \(g\mapsto \rho_g\) is a genuine representation of the group \(G\) on \(V\).
	\end{proposition}
	
	\begin{proof}
		Fix \(g\in G\) and \(\alpha,\beta,\gamma\in E\). Since \(G\) acts on \(E\), we may write
		\[
		\alpha=g\bullet(g^{-1}\bullet\alpha).
		\]
		Therefore,
		\[
		(g,\alpha)=(g,\gamma)\vdash(1,g^{-1}\bullet\alpha).
		\]
		Applying axiom \emph{(R2)} and then \emph{(R3)}, we obtain
		\[
		\rho_{(g,\alpha)}
		=
		\rho_{(g,\gamma)}\,\rho_{(1,g^{-1}\bullet\alpha)}
		=
		\rho_{(g,\gamma)}\,\id_V
		=
		\rho_{(g,\gamma)}.
		\]
		Since \(\gamma\) is arbitrary, this proves that \(\rho_{(g,\alpha)}\) depends only on \(g\).
		
		Now let \(g,h\in G\) and choose any \(\alpha,\beta\in E\). Then
		\[
		(g,\alpha)\vdash(h,\beta)=(gh,\,g\bullet\beta),
		\]
		so by \emph{(R2)},
		\[
		\rho_{gh}
		=
		\rho_{(gh,g\bullet\beta)}
		=
		\rho_{(g,\alpha)}\,\rho_{(h,\beta)}
		=
		\rho_g\rho_h.
		\]
		Thus \(g\mapsto \rho_g\) is a genuine representation of \(G\).
	\end{proof}
	
	\begin{remark}\label{rem:lambda-factorization}
		If we write
		\[
		L_\alpha:=\lambda_{(1,\alpha)}
		\qquad(\alpha\in E),
		\]
		then for every \(g\in G\) and \(\alpha\in E\) one has
		\[
		\lambda_{(g,\alpha)}=L_\alpha\,\rho_g.
		\]
		Indeed, for any \(\beta\in E\),
		\[
		(1,\alpha)\dashv(g,\beta)=(g,\alpha),
		\]
		and therefore axiom \emph{(R5)} gives
		\[
		\lambda_{(g,\alpha)}
		=
		\lambda_{(1,\alpha)\dashv(g,\beta)}
		=
		\lambda_{(1,\alpha)}\rho_{(g,\beta)}
		=
		L_\alpha\rho_g.
		\]
		We shall not use this identity immediately, but it will later clarify the role played by the halo.
	\end{remark}
	
	\subsection{Maschke splitting on the \texorpdfstring{$\rho$}{rho}-side}
	
	We now impose the hypothesis under which the classical Maschke theorem may be applied.
	
	\medskip
	\noindent
	\textbf{Standing hypothesis.}
	\emph{Throughout the rest of this section, we assume that \(G\) is finite and that
		\(\operatorname{char}(\kk)=0\).}
	\medskip
	
	Under this hypothesis, the group algebra \(\kk[G]\) is semisimple \cite[Theorem 8.7]{JamesLiebeck2001}. In particular, every
	\(\kk[G]\)-module is semisimple, and therefore every short exact sequence of \(G\)-modules splits.
	
	\begin{proposition}\label{prop:rho-splitting}
		Let
		\begin{equation}\label{eq:ses-main}
			0\longrightarrow W \stackrel{\iota}{\longrightarrow} V \stackrel{\pi}{\longrightarrow} Q\longrightarrow 0
		\end{equation}
		be a short exact sequence in \(\Rep(D)\). Then there exists a \(\kk\)-linear map
		\[
		s:Q\longrightarrow V
		\]
		such that
		\[
		\pi\circ s=\id_Q
		\]
		and
		\[
		\rho_x\circ s=s\circ \rho_x^Q
		\qquad(x\in D).
		\]
		Equivalently, every short exact sequence in \(\Rep(D)\) splits on the \(\rho\)-side.
	\end{proposition}
	
	\begin{proof}
		By Proposition~\ref{prop:rho-depends-only-on-g}, the operators \(\rho_x\) depend only on the group
		component \(g\in G\), and the assignment
		\[
		g\longmapsto \rho_g
		\]
		is a genuine representation of \(G\). Thus \eqref{eq:ses-main} is, in particular, a short exact
		sequence of \(G\)-modules.
		
		Since \(G\) is finite and \(\operatorname{char}(\kk)=0\), Maschke's theorem applies, and there
		exists a \(G\)-equivariant linear section
		\[
		s:Q\longrightarrow V
		\]
		such that
		\[
		\pi\circ s=\id_Q
		\]
		and
		\[
		\rho_g\circ s=s\circ \rho_g^Q
		\qquad(g\in G).
		\]
		Finally, because \(\rho_x\) depends only on the group component of \(x\in D=G\times E\), the same
		identity holds for every \(x\in D\). Thus \(s\) is \(\rho\)-equivariant on all of \(D\).
	\end{proof}
	
	Fix once and for all such a \(\rho\)-equivariant splitting \(s\). Then every element \(v\in V\)
	admits a unique decomposition
	\[
	v=\iota(w)+s(q)
	\qquad(w\in W,\ q\in Q),
	\]
	and hence \(V\) is identified, as a vector space, with \(W\oplus Q\).
	
	\begin{lemma}\label{lem:block-forms}
		With respect to the decomposition \(V\cong W\oplus Q\) induced by the splitting \(s\), the operators
		\(\rho_x\) and \(\lambda_x\) have the form
		\[
		\rho_x=
		\begin{pmatrix}
			\rho_x^W & 0\\
			0 & \rho_x^Q
		\end{pmatrix},
		\qquad
		\lambda_x=
		\begin{pmatrix}
			\lambda_x^W & \theta_x\\
			0 & \lambda_x^Q
		\end{pmatrix},
		\]
		where \(\theta_x\in \Hom_\kk(Q,W)\) for every \(x\in D\).
	\end{lemma}
	
	\begin{proof}
		Since \(W\) is a subrepresentation, it is stable under all \(\rho_x\) and all \(\lambda_x\). Hence
		the lower-left block of both operators is zero.
		
		Because the chosen splitting \(s\) is \(\rho\)-equivariant, one has
		\[
		\rho_x(s(q))=s(\rho_x^Q(q))
		\qquad(x\in D,\ q\in Q),
		\]
		so no \(W\)-component is produced from the \(Q\)-part under \(\rho_x\). Therefore the upper-right
		block of \(\rho_x\) is also zero. Thus
		\[
		\rho_x=
		\begin{pmatrix}
			\rho_x^W & 0\\
			0 & \rho_x^Q
		\end{pmatrix}.
		\]
		
		For \(\lambda_x\) there is a uniquely
		determined linear map
		\[
		\theta_x:Q\to W,
		\]
		which records the \(W\)-component of \(\lambda_x(s(q))\). Hence
		\[
		\lambda_x=
		\begin{pmatrix}
			\lambda_x^W & \theta_x\\
			0 & \lambda_x^Q
		\end{pmatrix}.
		\]
	\end{proof}
	
	\begin{remark}
		The family \(\theta=\{\theta_x\}_{x\in D}\) is the obstruction to extending the \(\rho\)-equivariant
		splitting to a full splitting in \(\Rep(D)\). If \(\theta_x=0\) for every \(x\in D\), then the image
		of \(s\) is stable under all \(\lambda_x\) and all \(\rho_x\), and therefore \eqref{eq:ses-main}
		splits in the full representation category.
	\end{remark}
	
	\subsection{\texorpdfstring{The cocycle space \(Z^1(Q,W)\)}{The cocycle space Z1(Q,W)}}
	
	We now determine the identities satisfied by the family \(\theta\).
	
	\begin{definition}\label{def:Z1}
		Let \(Q,W\in \Rep(D)\). Define \(Z^1(Q,W)\) to be the set of all families
		\[
		\theta=\{\theta_x\}_{x\in D},
		\qquad
		\theta_x\in \Hom_\kk(Q,W),
		\]
		such that, for all \(x,y\in D\),
		\begin{align}
			\theta_{x\dashv y}&=\lambda_x^W\theta_y+\theta_x\lambda_y^Q, \label{eq:Z1a}\\
			\theta_{x\vdash y}&=\rho_x^W\theta_y, \label{eq:Z1b}\\
			\theta_{x\dashv y}&=\theta_x\rho_y^Q. \label{eq:Z1c}
		\end{align}
	\end{definition}
	
	\begin{lemma}\label{lem:cocycle-identities}
		Let \eqref{eq:ses-main} be a short exact sequence in \(\Rep(D)\), choose a \(\rho\)-equivariant
		splitting \(s\), and let \(\theta=\{\theta_x\}_{x\in D}\) be the corresponding off-diagonal family
		from Lemma~\ref{lem:block-forms}. Then \(\theta\in Z^1(Q,W)\).
	\end{lemma}
	
	\begin{proof}
		We compare the upper-right blocks in the representation identities.
		
		\smallskip
		\noindent
		\textbf{Step 1: the identity \eqref{eq:Z1a}.}
		Axiom \emph{(R1)} gives
		\[
		\lambda_{x\dashv y}=\lambda_x\lambda_y.
		\]
		Using Lemma~\ref{lem:block-forms},
		\[
		\lambda_x=
		\begin{pmatrix}
			\lambda_x^W & \theta_x\\
			0 & \lambda_x^Q
		\end{pmatrix},
		\qquad
		\lambda_y=
		\begin{pmatrix}
			\lambda_y^W & \theta_y\\
			0 & \lambda_y^Q
		\end{pmatrix}.
		\]
		Multiplying these matrices yields
		\[
		\lambda_x\lambda_y=
		\begin{pmatrix}
			\lambda_x^W\lambda_y^W &
			\lambda_x^W\theta_y+\theta_x\lambda_y^Q\\
			0 &
			\lambda_x^Q\lambda_y^Q
		\end{pmatrix}.
		\]
		On the other hand,
		\[
		\lambda_{x\dashv y}=
		\begin{pmatrix}
			\lambda_{x\dashv y}^W & \theta_{x\dashv y}\\
			0 & \lambda_{x\dashv y}^Q
		\end{pmatrix}.
		\]
		Comparing the upper-right blocks gives
		\[
		\theta_{x\dashv y}=\lambda_x^W\theta_y+\theta_x\lambda_y^Q.
		\]
		
		\smallskip
		\noindent
		\textbf{Step 2: the identity \eqref{eq:Z1b}.}
		Axiom \emph{(R4)} gives
		\[
		\rho_x\lambda_y=\lambda_{x\vdash y}.
		\]
		Since \(\rho_x\) is block diagonal,
		\[
		\rho_x=
		\begin{pmatrix}
			\rho_x^W & 0\\
			0 & \rho_x^Q
		\end{pmatrix},
		\qquad
		\lambda_y=
		\begin{pmatrix}
			\lambda_y^W & \theta_y\\
			0 & \lambda_y^Q
		\end{pmatrix},
		\]
		so
		\[
		\rho_x\lambda_y=
		\begin{pmatrix}
			\rho_x^W\lambda_y^W & \rho_x^W\theta_y\\
			0 & \rho_x^Q\lambda_y^Q
		\end{pmatrix}.
		\]
		Comparing the upper-right blocks with
		\[
		\lambda_{x\vdash y}=
		\begin{pmatrix}
			\lambda_{x\vdash y}^W & \theta_{x\vdash y}\\
			0 & \lambda_{x\vdash y}^Q
		\end{pmatrix}
		\]
		gives
		\[
		\theta_{x\vdash y}=\rho_x^W\theta_y.
		\]
		
		\smallskip
		\noindent
		\textbf{Step 3: the identity \eqref{eq:Z1c}.}
		Axiom \emph{(R5)} gives
		\[
		\lambda_x\rho_y=\lambda_{x\dashv y}.
		\]
		Using the block forms,
		\[
		\lambda_x\rho_y=
		\begin{pmatrix}
			\lambda_x^W & \theta_x\\
			0 & \lambda_x^Q
		\end{pmatrix}
		\begin{pmatrix}
			\rho_y^W & 0\\
			0 & \rho_y^Q
		\end{pmatrix}
		=
		\begin{pmatrix}
			\lambda_x^W\rho_y^W & \theta_x\rho_y^Q\\
			0 & \lambda_x^Q\rho_y^Q
		\end{pmatrix}.
		\]
		Comparing the upper-right blocks with \(\lambda_{x\dashv y}\) yields
		\[
		\theta_{x\dashv y}=\theta_x\rho_y^Q.
		\]
		
		Thus \(\theta\) satisfies all defining identities of \(Z^1(Q,W)\).
	\end{proof}
	
	\subsection{\texorpdfstring{Coboundaries and the quotient \(Z^1/B^1\)}{Coboundaries and the quotient Z1/B1}}
	
	We next describe the dependence of \(\theta\) on the choice of \(\rho\)-equivariant splitting.
	
	\begin{definition}\label{def:B1}
		Let \(Q,W\in\Rep(D)\). Define
		\[
		\Hom_\rho(Q,W):=
		\{t\in\Hom_\kk(Q,W)\mid \rho_g^W t=t\rho_g^Q\text{ for all }g\in G\}.
		\]
		For \(t\in \Hom_\rho(Q,W)\), define a family
		\[
		\delta t=\{(\delta t)_x\}_{x\in D}
		\]
		by
		\begin{equation}\label{eq:delta}
			(\delta t)_x:=\lambda_x^W t-t\lambda_x^Q.
		\end{equation}
		Let \(B^1(Q,W)\) be the set of all families of the form \(\delta t\), \(t\in\Hom_\rho(Q,W)\).
	\end{definition}
	
	\begin{lemma}\label{lem:B1subsetZ1}
		For every \(Q,W\in\Rep(D)\), one has
		\[
		B^1(Q,W)\subseteq Z^1(Q,W).
		\]
	\end{lemma}
	
	\begin{proof}
		Let \(t\in \Hom_\rho(Q,W)\), and define \(\theta:=\delta t\) by \eqref{eq:delta}. We verify the
		three cocycle identities.
		
		\smallskip
		\noindent
		\textbf{Verification of \eqref{eq:Z1a}.}
		Using axiom \emph{(R1)} for \(W\) and \(Q\), we compute
		\begin{align*}
			\theta_{x\dashv y}
			&=\lambda_{x\dashv y}^W t-t\lambda_{x\dashv y}^Q\\
			&=\lambda_x^W\lambda_y^W t-t\lambda_x^Q\lambda_y^Q\\
			&=\lambda_x^W(\lambda_y^W t-t\lambda_y^Q)+(\lambda_x^W t-t\lambda_x^Q)\lambda_y^Q\\
			&=\lambda_x^W\theta_y+\theta_x\lambda_y^Q.
		\end{align*}
		
		\smallskip
		\noindent
		\textbf{Verification of \eqref{eq:Z1b}.}
		Using axiom \emph{(R4)} for \(W\) and \(Q\), and the fact that \(t\) intertwines the \(\rho\)-action,
		we obtain
		\begin{align*}
			\theta_{x\vdash y}
			&=\lambda_{x\vdash y}^W t-t\lambda_{x\vdash y}^Q\\
			&=\rho_x^W\lambda_y^W t-t\rho_x^Q\lambda_y^Q\\
			&=\rho_x^W\lambda_y^W t-\rho_x^W t\lambda_y^Q\\
			&=\rho_x^W(\lambda_y^W t-t\lambda_y^Q)\\
			&=\rho_x^W\theta_y.
		\end{align*}
		
		\smallskip
		\noindent
		\textbf{Verification of \eqref{eq:Z1c}.}
		Using axiom \emph{(R5)} for \(W\) and \(Q\), again together with the intertwining property of \(t\),
		we compute
		\begin{align*}
			\theta_{x\dashv y}
			&=\lambda_{x\dashv y}^W t-t\lambda_{x\dashv y}^Q\\
			&=\lambda_x^W\rho_y^W t-t\lambda_x^Q\rho_y^Q\\
			&=\lambda_x^W t\rho_y^Q-t\lambda_x^Q\rho_y^Q\\
			&=(\lambda_x^W t-t\lambda_x^Q)\rho_y^Q\\
			&=\theta_x\rho_y^Q.
		\end{align*}
		
		Therefore \(\theta\in Z^1(Q,W)\), and hence \(B^1(Q,W)\subseteq Z^1(Q,W)\).
	\end{proof}
	
	\begin{lemma}\label{lem:change-splitting}
		Let \eqref{eq:ses-main} be a short exact sequence in \(\Rep(D)\), let \(s:Q\to V\) be a
		\(\rho\)-equivariant splitting, and let \(\theta\in Z^1(Q,W)\) be the corresponding cocycle.
		For \(t\in\Hom_\rho(Q,W)\), define
		\[
		s_t:=s+\iota t.
		\]
		Then \(s_t\) is again \(\rho\)-equivariant, and the cocycle associated with \(s_t\) is
		\[
		\theta+\delta t.
		\]
	\end{lemma}
	
	\begin{proof}
		First observe that
		\[
		\pi s_t=\pi s+\pi\iota\,t=\id_Q,
		\]
		so \(s_t\) is again a linear section of \(\pi\). Next, for every \(g\in G\),
		\begin{align*}
			\rho_g s_t
			&=\rho_g s+\rho_g\iota t\\
			&=s\rho_g^Q+\iota\rho_g^W t\\
			&=s\rho_g^Q+\iota t\rho_g^Q\\
			&=(s+\iota t)\rho_g^Q\\
			&=s_t\rho_g^Q,
		\end{align*}
		because \(s\) is \(\rho\)-equivariant and \(t\in\Hom_\rho(Q,W)\). Since \(\rho_x\) depends only on
		the group component, this proves that \(s_t\) is \(\rho\)-equivariant for every \(x\in D\).
		
		Let \(\Phi_t:W\oplus Q\to W\oplus Q\) be the linear automorphism
		\[
		\Phi_t(w,q)=(w+t(q),\,q).
		\]
		Then the block matrix of \(\lambda_x\) with respect to the splitting \(s_t\) is
		\[
		\Phi_t^{-1}
		\begin{pmatrix}
			\lambda_x^W & \theta_x\\
			0 & \lambda_x^Q
		\end{pmatrix}
		\Phi_t.
		\]
		A direct computation gives
		\[
		\Phi_t^{-1}
		\begin{pmatrix}
			\lambda_x^W & \theta_x\\
			0 & \lambda_x^Q
		\end{pmatrix}
		\Phi_t
		=
		\begin{pmatrix}
			\lambda_x^W & \theta_x+\lambda_x^W t-t\lambda_x^Q\\
			0 & \lambda_x^Q
		\end{pmatrix}.
		\]
		Hence the new off-diagonal block is
		\[
		\theta_x+(\delta t)_x.
		\]
		Therefore the cocycle associated with \(s_t\) is \(\theta+\delta t\).
	\end{proof}
	
	\subsection{\texorpdfstring{The cocycle model for \(\Ext^1\)}{The cocycle model for Ext1}}
	
	We may now identify \(\Ext^1\) with the quotient \(Z^1/B^1\).
	
	\begin{theorem}\label{thm:Ext1-cocycle}
		Let \(Q,W\in\Rep(D)\). Under the standing hypothesis that \(G\) is finite and
		\(\operatorname{char}(\kk)=0\), there is a bijection
		\[
		\Ext^1_{\Rep(D)}(Q,W)\cong Z^1(Q,W)/B^1(Q,W).
		\]
		Moreover, the zero class corresponds exactly to split short exact sequences.
	\end{theorem}
	
	\begin{proof}
		By Corollary~\ref{cor:Ext-defined}, the group \(\Ext^1_{\Rep(D)}(Q,W)\) is defined in the abelian
		category \(\Rep(D)\), which has enough injectives and projectives. We use the Yoneda interpretation
		of \(\Ext^1\) (\cite[Section~3.4]{Weibel1994} or \cite{Yoneda1960}), namely that
		\(\Ext^1_{\Rep(D)}(Q,W)\) classifies short exact sequences
		\[
		0\to W\to V\to Q\to 0
		\]
		up to the usual equivalence of extensions.
		
		\smallskip
		\noindent
		\textbf{Step 1: from an extension to a cocycle class.}
		Let
		\[
		0\to W \stackrel{\iota}{\to} V \stackrel{\pi}{\to} Q\to 0
		\]
		be a short exact sequence in \(\Rep(D)\). By Proposition~\ref{prop:rho-splitting}, there exists a
		\(\rho\)-equivariant splitting \(s:Q\to V\). The corresponding family
		\[
		\theta=\{\theta_x\}_{x\in D}
		\]
		obtained from Lemma~\ref{lem:block-forms} belongs to \(Z^1(Q,W)\) by
		Lemma~\ref{lem:cocycle-identities}. We assign to the extension the class
		\[
		[\theta]\in Z^1(Q,W)/B^1(Q,W).
		\]
		
		\smallskip
		\noindent
		\textbf{Step 2: independence of the choice of splitting.}
		If \(s'\) is another \(\rho\)-equivariant splitting, then
		\[
		s'=s+\iota t
		\]
		for a unique \(t\in \Hom_\kk(Q,W)\). Since both \(s\) and \(s'\) are \(\rho\)-equivariant, the map
		\(t\) lies in \(\Hom_\rho(Q,W)\). By Lemma~\ref{lem:change-splitting}, the cocycle attached to \(s'\)
		is
		\[
		\theta+\delta t.
		\]
		Thus the class of the cocycle in \(Z^1(Q,W)/B^1(Q,W)\) is independent of the chosen
		\(\rho\)-equivariant splitting.
		
		\smallskip
		\noindent
		\textbf{Step 3: from a cocycle to an extension.}
		Conversely, let \(\theta\in Z^1(Q,W)\). Define a \(\kk\)-vector space
		\[
		V_\theta:=W\oplus Q
		\]
		and define operators on \(V_\theta\) by
		\[
		\rho_x=
		\begin{pmatrix}
			\rho_x^W & 0\\
			0 & \rho_x^Q
		\end{pmatrix},
		\qquad
		\lambda_x=
		\begin{pmatrix}
			\lambda_x^W & \theta_x\\
			0 & \lambda_x^Q
		\end{pmatrix}.
		\]
		We claim that these operators define a representation of \(D\).
		
		Indeed, \(\rho_x\in\Aut(V_\theta)\) because \(\rho_x^W\) and \(\rho_x^Q\) are invertible. Axiom
		\emph{(R2)} follows immediately from block diagonality and the fact that \(W\) and \(Q\) are already
		representations:
		\[
		\rho_{x\vdash y}
		=
		\begin{pmatrix}
			\rho_{x\vdash y}^W & 0\\
			0 & \rho_{x\vdash y}^Q
		\end{pmatrix}
		=
		\begin{pmatrix}
			\rho_x^W\rho_y^W & 0\\
			0 & \rho_x^Q\rho_y^Q
		\end{pmatrix}
		=
		\rho_x\rho_y.
		\]
		Similarly, \(\rho_e=\id\) follows from \(\rho_e^W=\id_W\) and \(\rho_e^Q=\id_Q\).
		
		For axiom \emph{(R1)}, we compute
		\[
		\lambda_x\lambda_y=
		\begin{pmatrix}
			\lambda_x^W\lambda_y^W &
			\lambda_x^W\theta_y+\theta_x\lambda_y^Q\\
			0 &
			\lambda_x^Q\lambda_y^Q
		\end{pmatrix}.
		\]
		Since \(\theta\in Z^1(Q,W)\), identity \eqref{eq:Z1a} gives
		\[
		\lambda_x^W\theta_y+\theta_x\lambda_y^Q=\theta_{x\dashv y},
		\]
		while \(W\) and \(Q\) satisfy \emph{(R1)}. Hence
		\[
		\lambda_x\lambda_y=
		\begin{pmatrix}
			\lambda_{x\dashv y}^W & \theta_{x\dashv y}\\
			0 & \lambda_{x\dashv y}^Q
		\end{pmatrix}
		=
		\lambda_{x\dashv y}.
		\]
		
		For axiom \emph{(R4)},
		\[
		\rho_x\lambda_y=
		\begin{pmatrix}
			\rho_x^W\lambda_y^W & \rho_x^W\theta_y\\
			0 & \rho_x^Q\lambda_y^Q
		\end{pmatrix}.
		\]
		Since \(\theta\) satisfies \eqref{eq:Z1b}, and \(W,Q\) satisfy \emph{(R4)}, we obtain
		\[
		\rho_x\lambda_y=
		\begin{pmatrix}
			\lambda_{x\vdash y}^W & \theta_{x\vdash y}\\
			0 & \lambda_{x\vdash y}^Q
		\end{pmatrix}
		=
		\lambda_{x\vdash y}.
		\]
		
		For axiom \emph{(R5)},
		\[
		\lambda_x\rho_y=
		\begin{pmatrix}
			\lambda_x^W\rho_y^W & \theta_x\rho_y^Q\\
			0 & \lambda_x^Q\rho_y^Q
		\end{pmatrix}.
		\]
		Using \eqref{eq:Z1c} and the fact that \(W,Q\) satisfy \emph{(R5)}, we deduce
		\[
		\lambda_x\rho_y=
		\begin{pmatrix}
			\lambda_{x\dashv y}^W & \theta_{x\dashv y}\\
			0 & \lambda_{x\dashv y}^Q
		\end{pmatrix}
		=
		\lambda_{x\dashv y}.
		\]
		Thus \(V_\theta\) is a representation of \(D\). The canonical maps
		\[
		0\to W\to W\oplus Q\to Q\to 0
		\]
		define a short exact sequence in \(\Rep(D)\).
		
		\smallskip
		\noindent
		\textbf{Step 4: passage to the quotient by \(B^1\).}
		If \(\theta'=\theta+\delta t\) for some \(t\in\Hom_\rho(Q,W)\), define
		\[
		\Phi_t:W\oplus Q\to W\oplus Q,
		\qquad
		\Phi_t(w,q)=(w+t(q),q).
		\]
		Then \(\Phi_t\) is the identity on the subobject \(W\) and induces the identity on the quotient
		\(Q\). Moreover, the \(\rho\)-operators are block diagonal, so \(\Phi_t\) trivially intertwines the
		\(\rho\)-actions. By the computation in Lemma~\ref{lem:change-splitting}, \(\Phi_t\) conjugates the
		\(\lambda\)-operators attached to \(\theta\) to those attached to \(\theta'\). Therefore
		\(\Phi_t\) is an isomorphism of extensions. Hence cocycles differing by an element of
		\(B^1(Q,W)\) define equivalent extensions.
		
		\smallskip
		\noindent
		\textbf{Step 5: the two constructions are inverse.}
		Start with an extension
		\[
		0\to W\to V\to Q\to 0
		\]
		and choose a \(\rho\)-equivariant splitting \(s\). The associated cocycle \(\theta\) records the
		off-diagonal block of the \(\lambda\)-operators with respect to the decomposition
		\(V\simeq W\oplus Q\) induced by \(s\). Reconstructing the extension from \(\theta\) yields exactly
		the same block matrices for \(\lambda\) and \(\rho\), and therefore an equivalent extension.
		
		Conversely, start with \(\theta\in Z^1(Q,W)\). In the extension
		\[
		0\to W\to V_\theta\to Q\to 0,
		\]
		the obvious section
		\[
		s_0:Q\to V_\theta,
		\qquad
		s_0(q)=(0,q),
		\]
		is \(\rho\)-equivariant, and the off-diagonal block of \(\lambda_x\) with respect to the resulting
		decomposition \(V_\theta=W\oplus s_0(Q)\) is precisely \(\theta_x\). Thus the cocycle recovered from
		this extension is exactly \(\theta\).
		
		Hence the two constructions are inverse to each other, and we obtain a bijection
		\[
		\Ext^1_{\Rep(D)}(Q,W)\cong Z^1(Q,W)/B^1(Q,W).
		\]
		
		\smallskip
		\noindent
		\textbf{Step 6: characterization of split extensions.}
		Suppose first that the extension corresponding to \(\theta\) is split in \(\Rep(D)\). Then there
		exists a splitting which is both \(\rho\)-equivariant and \(\lambda\)-equivariant. With respect to
		such a splitting, all off-diagonal blocks vanish, so the corresponding cocycle is zero. Hence
		\([\theta]=0\) in \(Z^1(Q,W)/B^1(Q,W)\).
		
		Conversely, assume that \([\theta]=0\). Then \(\theta=\delta t\) for some \(t\in\Hom_\rho(Q,W)\).
		By Lemma~\ref{lem:change-splitting}, replacing the chosen \(\rho\)-equivariant splitting \(s\) by
		\[
		s' = s-\iota t
		\]
		changes the cocycle from \(\theta\) to
		\[
		\theta-\delta t=0.
		\]
		Hence the new splitting is both \(\rho\)-equivariant and \(\lambda\)-equivariant, so the extension
		splits in \(\Rep(D)\).
		
		Therefore the zero class in \(Z^1(Q,W)/B^1(Q,W)\) corresponds exactly to split short exact
		sequences.
	\end{proof}
	
	\begin{remark}
		Theorem~\ref{thm:Ext1-cocycle} is the precise form of the failure of Maschke semisimplicity in the
		representation theory of generalized digroups. Classical Maschke theory annihilates the obstruction
		on the group side, but the off-diagonal \(\lambda\)-data survives and is measured by an explicit
		cohomology class.
	\end{remark}
	
	\section{Explicit nonsplit extensions}
	\label{sec:examples}
	
	The abstract cohomological description obtained in the previous section becomes more meaningful
	once one sees it in a concrete example. The purpose of this section is therefore to construct an
	explicit nonsplit extension in \(\Rep(D)\), showing that the category need not be semisimple even
	when the group side lies in a classical Maschke regime. This example also makes the cocycle
	description of \(\Ext^1\) completely tangible.
	
	\subsection{The underlying generalized digroup}
	
	Let
	\[
	G=C_2=\{1,s\},
	\qquad
	E=\{e_0,e_1\},
	\]
	and let \(G\) act trivially on \(E\), that is,
	\[
	g\bullet e_i=e_i
	\qquad
	(g\in G,\ i\in\{0,1\}).
	\]
	By Construction~\ref{cons:product-model} and Proposition~\ref{prop:product-gd}, the set
	\[
	D:=G\times E
	\]
	is a generalized digroup with operations
	\begin{equation}\label{eq:C2E-ops}
		(g,e_i)\vdash(h,e_j)=(gh,e_j),
		\qquad
		(g,e_i)\dashv(h,e_j)=(gh,e_i),
	\end{equation}
	and halo
	\[
	\overline{E}=\{(1,e_0),(1,e_1)\}.
	\]
	
	Since \(G\) is finite and \(\operatorname{char}(\kk)=0\), the standing hypothesis of
	Section~\ref{sec:cocycle} is satisfied.
	
	\subsection{A two-dimensional representation}
	
	Let \(V=\kk^2\), with ordered basis
	\[
	v_1=\binom{1}{0},
	\qquad
	v_2=\binom{0}{1}.
	\]
	Let
	\[
	\chi:G\to \kk^\times
	\]
	be the nontrivial character of \(C_2\), given by
	\[
	\chi(1)=1,
	\qquad
	\chi(s)=-1.
	\]
	
	Define two matrices
	\[
	P_0=
	\begin{pmatrix}
		1&0\\
		0&0
	\end{pmatrix},
	\qquad
	P_1=
	\begin{pmatrix}
		1&0\\
		1&0
	\end{pmatrix}.
	\]
	
	\begin{lemma}\label{lem:P-band}
		The matrices \(P_0\) and \(P_1\) satisfy
		\[
		P_iP_j=P_i
		\qquad
		(i,j\in\{0,1\}).
		\]
		In particular,
		\[
		P_0^2=P_0,
		\qquad
		P_1^2=P_1,
		\qquad
		P_0P_1=P_0,
		\qquad
		P_1P_0=P_1.
		\]
	\end{lemma}
	
	\begin{proof}
		A direct matrix computation gives
		\[
		P_0^2=
		\begin{pmatrix}
			1&0\\
			0&0
		\end{pmatrix}
		\begin{pmatrix}
			1&0\\
			0&0
		\end{pmatrix}
		=
		\begin{pmatrix}
			1&0\\
			0&0
		\end{pmatrix}
		=P_0,
		\]
		and
		\[
		P_1^2=
		\begin{pmatrix}
			1&0\\
			1&0
		\end{pmatrix}
		\begin{pmatrix}
			1&0\\
			1&0
		\end{pmatrix}
		=
		\begin{pmatrix}
			1&0\\
			1&0
		\end{pmatrix}
		=P_1.
		\]
		Similarly,
		\[
		P_0P_1=
		\begin{pmatrix}
			1&0\\
			0&0
		\end{pmatrix}
		\begin{pmatrix}
			1&0\\
			1&0
		\end{pmatrix}
		=
		\begin{pmatrix}
			1&0\\
			0&0
		\end{pmatrix}
		=P_0,
		\]
		and
		\[
		P_1P_0=
		\begin{pmatrix}
			1&0\\
			1&0
		\end{pmatrix}
		\begin{pmatrix}
			1&0\\
			0&0
		\end{pmatrix}
		=
		\begin{pmatrix}
			1&0\\
			1&0
		\end{pmatrix}
		=P_1.
		\]
		Hence \(P_iP_j=P_i\) for all \(i,j\in\{0,1\}\).
	\end{proof}
	
	We now define maps
	\[
	\rho:D\to \Aut(V),
	\qquad
	\lambda:D\to \End(V),
	\]
	by
	\begin{equation}\label{eq:rho-example}
		\rho_{(g,e_i)}:=\chi(g)\,I_2,
	\end{equation}
	and
	\begin{equation}\label{eq:lambda-example}
		\lambda_{(g,e_0)}:=\chi(g)\,P_0,
		\qquad
		\lambda_{(g,e_1)}:=\chi(g)\,P_1.
	\end{equation}
	
	\begin{proposition}\label{prop:example-representation}
		The triple \((V,\lambda,\rho)\) is a representation of the generalized digroup \(D\).
	\end{proposition}
	
	\begin{proof}
		We verify the five axioms of Definition~\ref{def:representation}.
		
		Let
		\[
		x=(g,e_i),
		\qquad
		y=(h,e_j)
		\]
		be arbitrary elements of \(D\), with \(i,j\in\{0,1\}\).
		
		\smallskip
		\noindent
		\textbf{Verification of \emph{(R1)}.}
		Since the action of \(G\) on \(E\) is trivial, \eqref{eq:C2E-ops} gives
		\[
		x\dashv y=(gh,e_i).
		\]
		Therefore
		\[
		\lambda_{x\dashv y}
		=
		\lambda_{(gh,e_i)}
		=
		\chi(gh)\,P_i.
		\]
		On the other hand,
		\[
		\lambda_x\lambda_y
		=
		(\chi(g)P_i)(\chi(h)P_j)
		=
		\chi(g)\chi(h)\,P_iP_j.
		\]
		By Lemma~\ref{lem:P-band},
		\[
		P_iP_j=P_i,
		\]
		hence
		\[
		\lambda_x\lambda_y
		=
		\chi(gh)\,P_i
		=
		\lambda_{x\dashv y}.
		\]
		
		\smallskip
		\noindent
		\textbf{Verification of \emph{(R2)}.}
		Again by \eqref{eq:C2E-ops},
		\[
		x\vdash y=(gh,e_j).
		\]
		Thus
		\[
		\rho_{x\vdash y}
		=
		\rho_{(gh,e_j)}
		=
		\chi(gh)\,I_2.
		\]
		Since \(\chi\) is a group character,
		\[
		\chi(gh)=\chi(g)\chi(h),
		\]
		and therefore
		\[
		\rho_x\rho_y
		=
		(\chi(g)I_2)(\chi(h)I_2)
		=
		\chi(g)\chi(h)\,I_2
		=
		\chi(gh)\,I_2
		=
		\rho_{x\vdash y}.
		\]
		
		\smallskip
		\noindent
		\textbf{Verification of \emph{(R3)}.}
		Let \((1,e_i)\in\overline{E}\). Then
		\[
		\rho_{(1,e_i)}=\chi(1)\,I_2=I_2=\id_V.
		\]
		
		\smallskip
		\noindent
		\textbf{Verification of \emph{(R4)}.}
		Since \(x\vdash y=(gh,e_j)\), we have
		\[
		\lambda_{x\vdash y}=\chi(gh)\,P_j.
		\]
		On the other hand,
		\[
		\rho_x\lambda_y
		=
		(\chi(g)I_2)(\chi(h)P_j)
		=
		\chi(gh)\,P_j
		=
		\lambda_{x\vdash y}.
		\]
		
		\smallskip
		\noindent
		\textbf{Verification of \emph{(R5)}.}
		Since \(x\dashv y=(gh,e_i)\), we have
		\[
		\lambda_{x\dashv y}=\chi(gh)\,P_i.
		\]
		On the other hand,
		\[
		\lambda_x\rho_y
		=
		(\chi(g)P_i)(\chi(h)I_2)
		=
		\chi(gh)\,P_i
		=
		\lambda_{x\dashv y}.
		\]
		
		All five axioms are satisfied. Hence \((V,\lambda,\rho)\) is a representation of \(D\).
	\end{proof}
	
	\subsection{A short exact sequence}
	
	Consider the one-dimensional subspace
	\[
	W:=\Span\{v_2\}\subseteq V.
	\]
	
	\begin{proposition}\label{prop:W-subrepresentation}
		The subspace \(W\) is a subrepresentation of \(V\).
	\end{proposition}
	
	\begin{proof}
		Let \(x=(g,e_i)\in D\). Since \(\rho_x=\chi(g)I_2\), one has
		\[
		\rho_x(v_2)=\chi(g)v_2\in W.
		\]
		
		It remains to check stability under \(\lambda_x\). If \(i=0\), then
		\[
		\lambda_{(g,e_0)}(v_2)
		=
		\chi(g)P_0\binom{0}{1}
		=
		\chi(g)\binom{0}{0}
		=
		0\in W.
		\]
		If \(i=1\), then
		\[
		\lambda_{(g,e_1)}(v_2)
		=
		\chi(g)P_1\binom{0}{1}
		=
		\chi(g)\binom{0}{0}
		=
		0\in W.
		\]
		Thus \(W\) is stable under all \(\lambda_x\) and all \(\rho_x\), so \(W\) is a subrepresentation.
	\end{proof}
	
	Let
	\[
	Q:=V/W
	\]
	and denote by
	\[
	\overline{v}_1:=v_1+W
	\]
	the class of \(v_1\) modulo \(W\). Then \(Q\) is one-dimensional, generated by \(\overline{v}_1\),
	and we have a short exact sequence in \(\Rep(D)\):
	\begin{equation}\label{eq:example-ses}
		0\longrightarrow W \stackrel{\iota}{\longrightarrow} V \stackrel{\pi}{\longrightarrow} Q\longrightarrow 0.
	\end{equation}
	
	\subsection{The extension is nonsplit}
	
	We now show that \eqref{eq:example-ses} does not split in \(\Rep(D)\).
	
	\begin{proposition}\label{prop:example-nonsplit}
		The short exact sequence \eqref{eq:example-ses} is nonsplit in \(\Rep(D)\).
	\end{proposition}
	
	\begin{proof}
		Suppose, for contradiction, that \eqref{eq:example-ses} splits in \(\Rep(D)\). Then there exists
		a one-dimensional subrepresentation \(U\subseteq V\) such that
		\[
		V=W\oplus U.
		\]
		Since \(W=\Span\{v_2\}\), every one-dimensional complement \(U\) of \(W\) is of the form
		\[
		U=\Span\{u\},
		\qquad
		u=v_1+a v_2
		\]
		for a unique scalar \(a\in\kk\).
		
		Because \(U\) is a subrepresentation, it must be stable under \(\lambda_{(1,e_1)}\). Now
		\[
		\lambda_{(1,e_1)}u
		=
		P_1\binom{1}{a}
		=
		\binom{1}{1}.
		\]
		Hence \(\lambda_{(1,e_1)}u\in U\), so there exists \(c\in\kk\) such that
		\[
		\binom{1}{1}=c\binom{1}{a}.
		\]
		Comparing the first coordinates gives \(c=1\), and therefore the second coordinates give
		\[
		a=1.
		\]
		Thus necessarily
		\[
		U=\Span\{v_1+v_2\}.
		\]
		
		We now impose stability under \(\lambda_{(1,e_0)}\). Since
		\[
		\lambda_{(1,e_0)}(v_1+v_2)
		=
		P_0\binom{1}{1}
		=
		\binom{1}{0}
		=
		v_1,
		\]
		stability would require \(v_1\in U\). But \(U=\Span\{v_1+v_2\}\), and \(v_1\) is not a scalar
		multiple of \(v_1+v_2\). This is impossible.
		
		The contradiction shows that no \(D\)-stable complement of \(W\) exists. Therefore
		\eqref{eq:example-ses} is nonsplit in \(\Rep(D)\).
	\end{proof}
	
	\subsection{\texorpdfstring{The corresponding class in \(\Ext^1\)}{The corresponding class in Ext1}}
	
	We finally identify the cocycle associated with the extension \eqref{eq:example-ses}.
	
	Fix the linear section
	\[
	s:Q\to V,
	\qquad
	s(\overline{v}_1)=v_1.
	\]
	Since \(\rho_{(g,e_i)}=\chi(g)I_2\) is scalar, this section is automatically \(\rho\)-equivariant.
	Thus Theorem~\ref{thm:Ext1-cocycle} applies.
	
	With respect to the decomposition
	\[
	V=W\oplus s(Q)=\Span\{v_2\}\oplus\Span\{v_1\},
	\]
	the induced actions on \(W\) and \(Q\) are as follows.
	
	\begin{lemma}\label{lem:diagonal-actions-example}
		For every \(g\in G\) and \(i\in\{0,1\}\),
		\[
		\rho_{(g,e_i)}^W=\chi(g)\id_W,
		\qquad
		\rho_{(g,e_i)}^Q=\chi(g)\id_Q,
		\]
		\[
		\lambda_{(g,e_i)}^W=0,
		\qquad
		\lambda_{(g,e_i)}^Q=\chi(g)\id_Q.
		\]
	\end{lemma}
	
	\begin{proof}
		The formulas for \(\rho^W\) and \(\rho^Q\) are immediate from the fact that \(\rho_{(g,e_i)}\)
		is the scalar matrix \(\chi(g)I_2\).
		
		For \(\lambda^W\), we have already seen in the proof of Proposition~\ref{prop:W-subrepresentation}
		that \(\lambda_{(g,e_i)}(v_2)=0\). Hence
		\[
		\lambda_{(g,e_i)}^W=0.
		\]
		
		For \(\lambda^Q\), we compute modulo \(W\). If \(i=0\), then
		\[
		\lambda_{(g,e_0)}(v_1)
		=
		\chi(g)P_0\binom{1}{0}
		=
		\chi(g)\binom{1}{0},
		\]
		so
		\[
		\lambda_{(g,e_0)}^Q(\overline{v}_1)=\chi(g)\overline{v}_1.
		\]
		If \(i=1\), then
		\[
		\lambda_{(g,e_1)}(v_1)
		=
		\chi(g)P_1\binom{1}{0}
		=
		\chi(g)\binom{1}{1},
		\]
		and since \(\binom{1}{1}\equiv \binom{1}{0}\pmod{W}\), it follows that
		\[
		\lambda_{(g,e_1)}^Q(\overline{v}_1)=\chi(g)\overline{v}_1.
		\]
		Thus, in both cases,
		\[
		\lambda_{(g,e_i)}^Q=\chi(g)\id_Q.
		\]
	\end{proof}
	
	Let
	\[
	\beta=\{\beta_x\}_{x\in D}
	\]
	be the off-diagonal family attached to the splitting \(s\). Thus
	\[
	\lambda_x=
	\begin{pmatrix}
		\lambda_x^W & \beta_x\\
		0 & \lambda_x^Q
	\end{pmatrix}
	\]
	with respect to the decomposition \(V=W\oplus s(Q)\).
	
	\begin{proposition}\label{prop:beta-example}
		The family \(\beta\) is given by
		\[
		\beta_{(g,e_0)}=0,
		\qquad
		\beta_{(g,e_1)}(\overline{v}_1)=\chi(g)v_2.
		\]
		In particular,
		\[
		\beta_{(1,e_1)}(\overline{v}_1)=v_2\neq 0.
		\]
		Hence the class of \eqref{eq:example-ses} in \(\Ext^1_{\Rep(D)}(Q,W)\) is nonzero.
	\end{proposition}
	
	\begin{proof}
		By definition, \(\beta_x\) is the \(W\)-component of \(\lambda_x(s(\overline{v}_1))\).
		
		If \(x=(g,e_0)\), then
		\[
		\lambda_{(g,e_0)}(v_1)
		=
		\chi(g)P_0\binom{1}{0}
		=
		\chi(g)\binom{1}{0}.
		\]
		This vector lies entirely in \(s(Q)=\Span\{v_1\}\), so its \(W\)-component is zero. Hence
		\[
		\beta_{(g,e_0)}=0.
		\]
		
		If \(x=(g,e_1)\), then
		\[
		\lambda_{(g,e_1)}(v_1)
		=
		\chi(g)P_1\binom{1}{0}
		=
		\chi(g)\binom{1}{1}
		=
		\chi(g)v_1+\chi(g)v_2.
		\]
		Thus the \(W\)-component is \(\chi(g)v_2\), so
		\[
		\beta_{(g,e_1)}(\overline{v}_1)=\chi(g)v_2.
		\]
		In particular,
		\[
		\beta_{(1,e_1)}(\overline{v}_1)=v_2\neq 0.
		\]
		
		By Theorem~\ref{thm:Ext1-cocycle}, the extension class of \eqref{eq:example-ses} corresponds to
		the class of \(\beta\) in \(Z^1(Q,W)/B^1(Q,W)\). If this class were zero, the extension would split
		in \(\Rep(D)\). This contradicts Proposition~\ref{prop:example-nonsplit}. Therefore the class is
		nonzero.
	\end{proof}
	
	\begin{remark}
		Proposition~\ref{prop:beta-example} gives a completely explicit obstruction to Maschke
		semisimplicity in the category \(\Rep(D)\). The \(\rho\)-side splits, as predicted by
		Section~\ref{sec:cocycle}, but the off-diagonal \(\lambda\)-data survives and defines a nontrivial
		class in \(\Ext^1_{\Rep(D)}(Q,W)\).
	\end{remark}
	
	\section{A spectral sequence for generalized digroups}
	\label{sec:spectral}
	
	Having obtained an explicit cocycle model for first extensions, we now move to a more global
	cohomological viewpoint. The purpose of this section is to separate the group-theoretic and
	halo-theoretic contributions to extension theory by means of a spectral sequence. This provides a
	conceptual explanation for the phenomenon already visible in the previous sections: the group side
	may become homologically trivial under a Maschke hypothesis, while the halo side continues to
	support nontrivial extension classes.
	
	Throughout this section, fix a structural presentation
	\[
	D \simeq G\times E
	\]
	as in Theorem~\ref{thm:structure}. We write the operations as
	\[
	(g,\alpha)\vdash(h,\beta)=(gh,\,g\bullet\beta),
	\qquad
	(g,\alpha)\dashv(h,\beta)=(gh,\,\alpha).
	\]
	
	\subsection{The halo algebra}
	
	We first isolate the algebra generated by the halo.
	
	\begin{definition}\label{def:BE}
		Let \(B_E\) be the unital associative \(\kk\)-algebra generated by symbols
		\[
		\{\varepsilon_\alpha \mid \alpha\in E\}
		\]
		subject to the defining relations
		\begin{equation}\label{eq:BE-rel}
			\varepsilon_\alpha\varepsilon_\beta=\varepsilon_\alpha
			\qquad
			(\alpha,\beta\in E).
		\end{equation}
	\end{definition}
	
	\begin{remark}
		Since \(B_E\) is defined as a quotient of a free associative algebra, associativity is automatic.
		Relation \eqref{eq:BE-rel} implies that every \(\varepsilon_\alpha\) is idempotent and that the
		product of two generators remembers only the left factor. Thus \(B_E\) algebraically records the
		left-zero behaviour of the halo.
	\end{remark}
	
	The action of \(G\) on \(E\) induces a family of algebra automorphisms of \(B_E\).
	
	\begin{lemma}\label{lem:tau-action}
		For every \(g\in G\), the assignment
		\[
		\tau_g(\varepsilon_\alpha):=\varepsilon_{g\bullet\alpha}
		\qquad
		(\alpha\in E)
		\]
		extends uniquely to an algebra automorphism of \(B_E\). Moreover,
		\[
		\tau_{gh}=\tau_g\circ\tau_h
		\qquad
		(g,h\in G).
		\]
	\end{lemma}
	
	\begin{proof}
		Since \(B_E\) is defined by generators and relations, it is enough to verify that the assignment on
		generators preserves the defining relations. For \(\alpha,\beta\in E\),
		\[
		\tau_g(\varepsilon_\alpha\varepsilon_\beta)
		=
		\tau_g(\varepsilon_\alpha)
		=
		\varepsilon_{g\bullet\alpha},
		\]
		whereas
		\[
		\tau_g(\varepsilon_\alpha)\tau_g(\varepsilon_\beta)
		=
		\varepsilon_{g\bullet\alpha}\,\varepsilon_{g\bullet\beta}
		=
		\varepsilon_{g\bullet\alpha}.
		\]
		Hence \(\tau_g\) extends uniquely to an algebra endomorphism of \(B_E\).
		
		It is invertible, with inverse \(\tau_{g^{-1}}\), because
		\[
		\tau_{g^{-1}}(\tau_g(\varepsilon_\alpha))
		=
		\tau_{g^{-1}}(\varepsilon_{g\bullet\alpha})
		=
		\varepsilon_{g^{-1}\bullet(g\bullet\alpha)}
		=
		\varepsilon_\alpha.
		\]
		
		Finally, for \(\alpha\in E\),
		\[
		(\tau_g\circ\tau_h)(\varepsilon_\alpha)
		=
		\tau_g(\varepsilon_{h\bullet\alpha})
		=
		\varepsilon_{g\bullet(h\bullet\alpha)}
		=
		\varepsilon_{(gh)\bullet\alpha}
		=
		\tau_{gh}(\varepsilon_\alpha),
		\]
		because \(\bullet\) is a left action of \(G\) on \(E\).
	\end{proof}
	
	\subsection{The auxiliary semilinear category}
	
	We now introduce the auxiliary category that isolates the halo module structure together with the
	genuine action of the group component.
	
	\begin{definition}\label{def:semi-linear-cat}
		Let \(\mathcal C(E,G)\) be the category whose objects are pairs
		\[
		(M,t),
		\]
		where \(M\) is a left \(B_E\)-module and
		\[
		t_g\in\Aut_\kk(M)
		\qquad(g\in G)
		\]
		is a family satisfying
		\begin{align}
			t_1&=\id_M, \label{eq:C1}\\
			t_gt_h&=t_{gh}, \label{eq:C2}\\
			t_g(bm)&=\tau_g(b)\,t_g(m)
			\qquad
			(b\in B_E,\ m\in M). \label{eq:C3}
		\end{align}
		A morphism
		\[
		f:(M,t)\longrightarrow (N,s)
		\]
		in \(\mathcal C(E,G)\) is a \(B_E\)-linear map \(f:M\to N\) such that
		\[
		ft_g=s_gf
		\qquad(g\in G).
		\]
	\end{definition}
	
	We now relate \(\Rep(D)\) to \(\mathcal C(E,G)\).
	
	\begin{proposition}\label{prop:Rep-to-semi-linear}
		Let \((V,\lambda,\rho)\in \Rep(D)\). Define a left \(B_E\)-module structure on \(V\) by
		\[
		\varepsilon_\alpha\cdot v:=\lambda_{(1,\alpha)}(v)
		\qquad
		(\alpha\in E,\ v\in V),
		\]
		and define
		\[
		t_g:=\rho_g
		\qquad(g\in G).
		\]
		Then \((V,t)\) is an object of \(\mathcal C(E,G)\).
	\end{proposition}
	
	\begin{proof}
		First,
		\[
		\lambda_{(1,\alpha)}\lambda_{(1,\beta)}
		=
		\lambda_{(1,\alpha)\dashv(1,\beta)}
		=
		\lambda_{(1,\alpha)},
		\]
		so the assignment \(\varepsilon_\alpha\mapsto \lambda_{(1,\alpha)}\) defines a left
		\(B_E\)-module structure.
		
		Next, by Proposition~\ref{prop:rho-depends-only-on-g}, one has
		\[
		t_1=\rho_1=\id_V
		\qquad\text{and}\qquad
		t_gt_h=\rho_g\rho_h=\rho_{gh}=t_{gh}.
		\]
		Hence \eqref{eq:C1} and \eqref{eq:C2} hold.
		
		It remains to prove \eqref{eq:C3}. Let \(\alpha\in E\) and \(v\in V\). By axiom \emph{(R4)},
		\[
		\rho_g\lambda_{(1,\alpha)}
		=
		\lambda_{(g,g\bullet\alpha)}.
		\]
		Using Remark~\ref{rem:lambda-factorization}, we rewrite the right-hand side as
		\[
		\lambda_{(g,g\bullet\alpha)}
		=
		\lambda_{(1,g\bullet\alpha)}\rho_g.
		\]
		Hence
		\[
		\rho_g\lambda_{(1,\alpha)}
		=
		\lambda_{(1,g\bullet\alpha)}\rho_g.
		\]
		Equivalently,
		\[
		t_g(\varepsilon_\alpha\cdot v)
		=
		\varepsilon_{g\bullet\alpha}\cdot t_g(v)
		=
		\tau_g(\varepsilon_\alpha)\cdot t_g(v).
		\]
		This proves \eqref{eq:C3}.
	\end{proof}
	
	\begin{proposition}\label{prop:semi-linear-to-Rep}
		Let \((M,t)\in \mathcal C(E,G)\). Define maps
		\[
		\rho:D=G\times E\longrightarrow \Aut_\kk(M),
		\qquad
		\lambda:D=G\times E\longrightarrow \End_\kk(M)
		\]
		by
		\[
		\rho_{(g,\alpha)}:=t_g,
		\qquad
		\lambda_{(g,\alpha)}(m):=\varepsilon_\alpha\cdot t_g(m).
		\]
		Then \((M,\lambda,\rho)\) is a representation of \(D\). Moreover, this construction is inverse
		to that of Proposition~\ref{prop:Rep-to-semi-linear}. Consequently, the categories
		\(\Rep(D)\) and \(\mathcal C(E,G)\) are equivalent.
	\end{proposition}
	
	\begin{proof}
		We verify the five axioms of Definition~\ref{def:representation}.
		
		Let
		\[
		x=(g,\alpha),
		\qquad
		y=(h,\beta)
		\]
		be elements of \(D=G\times E\), and let \(m\in M\).
		
		\smallskip
		\noindent
		\textbf{Verification of \emph{(R1)}.}
		Since
		\[
		x\dashv y=(gh,\alpha),
		\]
		we have
		\[
		\lambda_{x\dashv y}(m)=\varepsilon_\alpha\cdot t_{gh}(m).
		\]
		On the other hand,
		\begin{align*}
			\lambda_x(\lambda_y(m))
			&=\lambda_x(\varepsilon_\beta\cdot t_h(m))\\
			&=\varepsilon_\alpha\cdot t_g(\varepsilon_\beta\cdot t_h(m))\\
			&=\varepsilon_\alpha\cdot \tau_g(\varepsilon_\beta)\,t_g t_h(m)\\
			&=\varepsilon_\alpha\cdot \varepsilon_{g\bullet\beta}\,t_{gh}(m)\\
			&=\varepsilon_\alpha\cdot t_{gh}(m),
		\end{align*}
		because \(\varepsilon_\alpha\varepsilon_{g\bullet\beta}=\varepsilon_\alpha\). Therefore
		\[
		\lambda_{x\dashv y}=\lambda_x\lambda_y.
		\]
		
		\smallskip
		\noindent
		\textbf{Verification of \emph{(R2)}.}
		Since
		\[
		x\vdash y=(gh,g\bullet\beta),
		\]
		we have
		\[
		\rho_{x\vdash y}=t_{gh}=t_gt_h=\rho_x\rho_y.
		\]
		
		\smallskip
		\noindent
		\textbf{Verification of \emph{(R3)}.}
		If \(e=(1,\alpha)\in\overline E\), then
		\[
		\rho_e=t_1=\id_M.
		\]
		
		\smallskip
		\noindent
		\textbf{Verification of \emph{(R4)}.}
		Using \eqref{eq:C3}, we compute
		\begin{align*}
			\rho_x(\lambda_y(m))
			&=t_g(\varepsilon_\beta\cdot t_h(m))\\
			&=\tau_g(\varepsilon_\beta)\,t_gt_h(m)\\
			&=\varepsilon_{g\bullet\beta}\cdot t_{gh}(m)\\
			&=\lambda_{(gh,g\bullet\beta)}(m)\\
			&=\lambda_{x\vdash y}(m).
		\end{align*}
		
		\smallskip
		\noindent
		\textbf{Verification of \emph{(R5)}.}
		We compute
		\begin{align*}
			\lambda_x(\rho_y(m))
			&=\varepsilon_\alpha\cdot t_g(t_h(m))\\
			&=\varepsilon_\alpha\cdot t_{gh}(m)\\
			&=\lambda_{(gh,\alpha)}(m)\\
			&=\lambda_{x\dashv y}(m).
		\end{align*}
		
		Thus \((M,\lambda,\rho)\) is a representation of \(D\).
		
		We now verify that this construction is inverse to that of
		Proposition~\ref{prop:Rep-to-semi-linear}.
		
		Start with \((V,\lambda,\rho)\in\Rep(D)\). Proposition~\ref{prop:Rep-to-semi-linear} defines
		the \(B_E\)-action by
		\[
		\varepsilon_\alpha\cdot v=\lambda_{(1,\alpha)}(v)
		\]
		and \(t_g=\rho_g\). Applying the present construction then gives
		\[
		\rho'_{(g,\alpha)}=t_g=\rho_g=\rho_{(g,\alpha)},
		\]
		because \(\rho_{(g,\alpha)}\) depends only on \(g\), and
		\[
		\lambda'_{(g,\alpha)}(v)
		=
		\varepsilon_\alpha\cdot t_g(v)
		=
		\lambda_{(1,\alpha)}\rho_g(v)
		=
		\lambda_{(g,\alpha)}(v)
		\]
		by Remark~\ref{rem:lambda-factorization}.
		
		Conversely, start with \((M,t)\in\mathcal C(E,G)\). The present construction defines
		\[
		\rho_{(g,\alpha)}=t_g,
		\qquad
		\lambda_{(g,\alpha)}(m)=\varepsilon_\alpha\cdot t_g(m).
		\]
		Applying Proposition~\ref{prop:Rep-to-semi-linear}, the induced \(B_E\)-action is
		\[
		\varepsilon_\alpha\cdot m=\lambda_{(1,\alpha)}(m)=\varepsilon_\alpha\cdot t_1(m)=\varepsilon_\alpha\cdot m,
		\]
		so it is the original one, and the induced operators are
		\[
		t'_g=\rho_g=t_g.
		\]
		Thus both constructions are mutually inverse on objects, and the same is immediate for morphisms.
		
		Therefore \(\Rep(D)\) and \(\mathcal C(E,G)\) are equivalent categories.
	\end{proof}
	
	\subsection{The $G$-action on $\Hom_{B_E}(Q,W)$}
	
	Let \(Q,W\in \Rep(D)\). By Propositions~\ref{prop:Rep-to-semi-linear}
	and~\ref{prop:semi-linear-to-Rep}, we regard them as objects
	\[
	(Q,t^Q),\qquad (W,t^W)
	\]
	of \(\mathcal C(E,G)\). In particular, \(Q\) and \(W\) are left \(B_E\)-modules, where
	\[
	\varepsilon_\alpha\cdot q=\lambda_{(1,\alpha)}^Q(q),
	\qquad
	\varepsilon_\alpha\cdot w=\lambda_{(1,\alpha)}^W(w),
	\]
	and they are endowed with genuine representations
	\[
	t_g^Q=\rho_g^Q,
	\qquad
	t_g^W=\rho_g^W
	\qquad(g\in G),
	\]
	satisfying
	\[
	t_g^Q t_h^Q=t_{gh}^Q,
	\qquad
	t_g^W t_h^W=t_{gh}^W.
	\]
	
	\begin{definition}\label{def:G-action-Hom}
		For \(f\in \Hom_{B_E}(Q,W)\) and \(g\in G\), define
		\begin{equation}\label{eq:GactionHom}
			(g\cdot f)(q):=
			t_g^W\bigl(f((t_g^Q)^{-1}(q))\bigr)
			\qquad
			(q\in Q).
		\end{equation}
	\end{definition}
	
	\begin{lemma}\label{lem:G-action-Hom}
		Formula \eqref{eq:GactionHom} defines a left action of \(G\) on \(\Hom_{B_E}(Q,W)\).
	\end{lemma}
	
	\begin{proof}
		Let \(f\in\Hom_{B_E}(Q,W)\), \(g,h\in G\), and \(q\in Q\). Then
		\begin{align*}
			\bigl(g\cdot(h\cdot f)\bigr)(q)
			&=
			t_g^W\Bigl((h\cdot f)\bigl((t_g^Q)^{-1}(q)\bigr)\Bigr)\\
			&=
			t_g^W t_h^W
			\Bigl(f\bigl((t_h^Q)^{-1}(t_g^Q)^{-1}(q)\bigr)\Bigr)\\
			&=
			t_{gh}^W
			\Bigl(f\bigl((t_{gh}^Q)^{-1}(q)\bigr)\Bigr)\\
			&=
			((gh)\cdot f)(q).
		\end{align*}
		Also,
		\[
		(1\cdot f)(q)=t_1^W(f((t_1^Q)^{-1}(q)))=f(q),
		\]
		so \(1\cdot f=f\).
		
		It remains to check that \(g\cdot f\) is again \(B_E\)-linear. Let \(b\in B_E\) and \(q\in Q\).
		Since \(t_g^Q\) is semilinear with twisting \(\tau_g\), its inverse is semilinear with twisting
		\(\tau_{g^{-1}}\). Therefore
		\[
		(t_g^Q)^{-1}(bq)=\tau_{g^{-1}}(b)\,(t_g^Q)^{-1}(q).
		\]
		Hence
		\begin{align*}
			(g\cdot f)(bq)
			&=
			t_g^W\Bigl(f\bigl((t_g^Q)^{-1}(bq)\bigr)\Bigr)\\
			&=
			t_g^W\Bigl(f\bigl(\tau_{g^{-1}}(b)\,(t_g^Q)^{-1}(q)\bigr)\Bigr)\\
			&=
			t_g^W\Bigl(\tau_{g^{-1}}(b)\,f((t_g^Q)^{-1}(q))\Bigr)\\
			&=
			b\,t_g^W\Bigl(f((t_g^Q)^{-1}(q))\Bigr)\\
			&=
			b\,(g\cdot f)(q).
		\end{align*}
		Hence \(g\cdot f\in\Hom_{B_E}(Q,W)\).
	\end{proof}
	
	\begin{lemma}\label{lem:Hom-invariants}
		One has
		\[
		\Hom_{\Rep(D)}(Q,W)=\Hom_{B_E}(Q,W)^G.
		\]
	\end{lemma}
	
	\begin{proof}
		Assume first that \(f\in \Hom_{\Rep(D)}(Q,W)\). Then \(f\) commutes with every
		\(\lambda_{(1,\alpha)}\), so \(f\) is \(B_E\)-linear. It also commutes with every \(\rho_g\), hence
		with every \(t_g\). Therefore
		\[
		(g\cdot f)(q)
		=
		t_g^W\bigl(f((t_g^Q)^{-1}(q))\bigr)
		=
		t_g^W\bigl((t_g^W)^{-1}(f(q))\bigr)
		=
		f(q),
		\]
		so \(f\in \Hom_{B_E}(Q,W)^G\).
		
		Conversely, let \(f\in \Hom_{B_E}(Q,W)^G\). The invariance condition means that
		\[
		t_g^W f (t_g^Q)^{-1}=f
		\qquad(g\in G).
		\]
		Equivalently,
		\[
		t_g^W f = f t_g^Q
		\qquad(g\in G).
		\]
		Since \(t_g=\rho_g\), this shows that \(f\) commutes with all \(\rho_g\), hence with all
		\(\rho_x\), because \(\rho_x\) depends only on the group component by
		Proposition~\ref{prop:rho-depends-only-on-g}. It also commutes with all
		\(\lambda_{(1,\alpha)}\), because it is \(B_E\)-linear. Finally, by
		Remark~\ref{rem:lambda-factorization},
		\[
		\lambda_{(g,\alpha)}=\lambda_{(1,\alpha)}\rho_g.
		\]
		Hence \(f\) commutes with every \(\lambda_{(g,\alpha)}\) as well. Therefore
		\(f\in \Hom_{\Rep(D)}(Q,W)\).
	\end{proof}
	
	\subsection{The halo functor and the spectral sequence}
	
	Fix \(Q\in \Rep(D)\). Define
	\[
	\mathcal F(-):=\Hom_{B_E}(Q,-),
	\]
	where the target is the category of left \(G\)-modules endowed with the action of
	Definition~\ref{def:G-action-Hom}.
	
	\begin{lemma}\label{lem:F-left-exact}
		The functor \(\mathcal F:\Rep(D)\to \kk G\text{-}\mathsf{Mod}\) is left exact.
	\end{lemma}
	
	\begin{proof}
		The functor \(\Hom_{B_E}(Q,-)\) is left exact on the underlying category of \(B_E\)-modules, and the
		\(G\)-action defined in \eqref{eq:GactionHom} is functorial. Hence \(\mathcal F\) is left exact.
	\end{proof}
	
	\begin{proposition}\label{prop:F-acyclic-Maschke}
		Assume that \(G\) is finite and \(\operatorname{char}(\kk)=0\). Then, for every injective
		object \(I\in\Rep(D)\), the \(G\)-module
		\[
		\mathcal F(I)=\Hom_{B_E}(Q,I)
		\]
		is \((-)^G\)-acyclic.
	\end{proposition}
	
	\begin{proof}
		Since \(G\) is finite and \(\operatorname{char}(\kk)=0\), the group algebra \(\kk G\) is
		semisimple. Hence
		\[
		H^p(G,M)=0
		\qquad
		(p>0)
		\]
		for every left \(\kk G\)-module \(M\). Therefore every \(G\)-module is \((-)^G\)-acyclic.
		In particular, \(\mathcal F(I)\) is \((-)^G\)-acyclic for every injective
		\(I\in\Rep(D)\).
	\end{proof}
	
	\begin{theorem}\label{thm:spectral}
		Assume that \(G\) is finite and \(\operatorname{char}(\kk)=0\).
		Let \(D\) be a generalized digroup and let \(Q,W\in\Rep(D)\).
		Then there exists a first-quadrant spectral sequence
		\[
		E_2^{p,q}
		\cong
		H^p\!\bigl(G,R^q\mathcal F(W)\bigr)
		\Longrightarrow
		\Ext_{\Rep(D)}^{p+q}(Q,W),
		\]
		where \(\mathcal F(W)=\Hom_{B_E}(Q,W)\) is endowed with the \(G\)-action of
		Definition~\ref{def:G-action-Hom}.
	\end{theorem}
	
	\begin{proof}
		By Lemma~\ref{lem:Hom-invariants}, the composite of \(\mathcal F\) with the invariants functor
		\[
		(-)^G:\kk G\text{-}\mathsf{Mod}\to \mathsf{Vect}_\kk
		\]
		is precisely
		\[
		W\longmapsto \Hom_{\Rep(D)}(Q,W).
		\]
		The invariants functor \((-)^G\) is left exact, and its right derived functors are the cohomology functors \(H^p(G,-)\); see Brown, Chapter III, especially the preliminary discussion. By proposition \ref{prop:F-acyclic-Maschke}, the Grothendieck spectral sequence (\cite[Theorem 10.47]{Rotman2009}) applies to
		the composite \((-)^G\circ \mathcal F\), and yields
		\[
		E_2^{p,q}
		=
		H^p\!\bigl(G,R^q\mathcal F(W)\bigr)
		\Longrightarrow
		R^{p+q}\!\bigl((-)^G\circ \mathcal F\bigr)(W).
		\]
		Since \(\Rep(D)\) is abelian with enough injectives by
		Corollaries~\ref{cor:abelian} and~\ref{cor:injectives}, the right derived functors of
		\[
		W\longmapsto \Hom_{\Rep(D)}(Q,W)
		\]
		are precisely \(\Ext_{\Rep(D)}^{p+q}(Q,W)\). Therefore
		\[
		E_2^{p,q}
		\cong
		H^p\!\bigl(G,R^q\mathcal F(W)\bigr)
		\Longrightarrow
		\Ext_{\Rep(D)}^{p+q}(Q,W).
		\]
	\end{proof}
	
	\begin{proposition}\label{prop:U-preserves-injectives}
		The forgetful functor
		\[
		U:\Rep(D)\longrightarrow B_E\text{-}\mathsf{Mod},
		\]
		obtained via Propositions~\ref{prop:Rep-to-semi-linear}
		and~\ref{prop:semi-linear-to-Rep}, sends injective objects of \(\Rep(D)\) to injective
		\(B_E\)-modules.
	\end{proposition}
	
	\begin{proof}
		By Propositions~\ref{prop:Rep-to-semi-linear} and~\ref{prop:semi-linear-to-Rep}, the category
		\(\Rep(D)\) is equivalent to \(\mathcal C(E,G)\). It is therefore enough to prove the claim for
		the forgetful functor
		\[
		U:\mathcal C(E,G)\longrightarrow B_E\text{-}\mathsf{Mod}.
		\]
		
		Define a functor
		\[
		L:B_E\text{-}\mathsf{Mod}\longrightarrow \mathcal C(E,G)
		\]
		as follows. For a left \(B_E\)-module \(M\), let
		\[
		L(M):=\bigoplus_{g\in G} M_g,
		\]
		where each \(M_g\) is a copy of \(M\) as a \(\kk\)-vector space, endowed with the twisted
		\(B_E\)-action
		\[
		b\cdot m_g := \bigl(\tau_{g^{-1}}(b)m\bigr)_g
		\qquad
		(b\in B_E,\ m\in M).
		\]
		Define operators
		\[
		t_h:L(M)\to L(M)
		\qquad(h\in G)
		\]
		on homogeneous elements by
		\[
		t_h(m_g):=m_{hg}.
		\]
		We check that \(L(M)\) is an object of \(\mathcal C(E,G)\). Clearly \(t_1=\id\) and
		\(t_h t_k=t_{hk}\). Moreover,
		\[
		t_h(b\cdot m_g)
		=
		t_h\bigl((\tau_{g^{-1}}(b)m)_g\bigr)
		=
		(\tau_{g^{-1}}(b)m)_{hg},
		\]
		while
		\[
		\tau_h(b)\cdot t_h(m_g)
		=
		\tau_h(b)\cdot m_{hg}
		=
		\bigl(\tau_{(hg)^{-1}}(\tau_h(b))m\bigr)_{hg}
		=
		(\tau_{g^{-1}}(b)m)_{hg},
		\]
		so the semilinearity condition holds.
		
		We now show that \(L\) is left adjoint to \(U\). Let \(M\in B_E\text{-}\mathsf{Mod}\) and
		\((N,t)\in\mathcal C(E,G)\). Given a morphism
		\[
		\Phi:L(M)\to (N,t)
		\]
		in \(\mathcal C(E,G)\), define
		\[
		f_\Phi:M\to U(N),
		\qquad
		f_\Phi(m):=\Phi(m_1).
		\]
		Since \(\Phi\) is \(B_E\)-linear, \(f_\Phi\) is \(B_E\)-linear.
		
		Conversely, given a \(B_E\)-linear map
		\[
		f:M\to U(N),
		\]
		define
		\[
		\widetilde f:L(M)\to N
		\]
		on homogeneous elements by
		\[
		\widetilde f(m_g):=t_g(f(m)).
		\]
		This map is \(B_E\)-linear, because
		\[
		\widetilde f(b\cdot m_g)
		=
		t_g\bigl(f(\tau_{g^{-1}}(b)m)\bigr)
		=
		t_g\bigl(\tau_{g^{-1}}(b)f(m)\bigr)
		=
		b\,t_g(f(m))
		=
		b\,\widetilde f(m_g),
		\]
		and it is \(G\)-equivariant by construction. The assignments
		\[
		\Phi\longmapsto f_\Phi,
		\qquad
		f\longmapsto \widetilde f
		\]
		are inverse to one another, so
		\[
		\Hom_{\mathcal C(E,G)}(L(M),(N,t))
		\cong
		\Hom_{B_E}(M,U(N))
		\]
		naturally.
		
		Hence \(L\) is left adjoint to \(U\). Finally, \(L\) is exact: twisting a module structure by an algebra
		automorphism is exact, and direct sums are exact in module categories. Therefore \(U\), being a
		right adjoint to an exact functor, preserves injective objects.
	\end{proof}
	
	\begin{lemma}\label{lem:U-exact}
		The forgetful functor
		\[
		U:\mathcal C(E,G)\longrightarrow B_E\text{-}\mathsf{Mod}
		\]
		is exact. Consequently, via the equivalence
		\(\Rep(D)\simeq\mathcal C(E,G)\), the functor
		\[
		U:\Rep(D)\longrightarrow B_E\text{-}\mathsf{Mod}
		\]
		is exact.
	\end{lemma}
	
	\begin{proof}
		Kernels and cokernels in \(\mathcal C(E,G)\) are computed on the underlying
		\(B_E\)-modules, endowed with the induced \(G\)-action by restriction and passage to
		quotients. Therefore a sequence in \(\mathcal C(E,G)\) is exact if and only if it is
		exact after forgetting the \(G\)-action. Hence \(U\) is exact.
	\end{proof}
	
	The next proposition identifies the \(q\)-direction with genuine halo cohomology
	by using the fact that the forgetful functor preserves injective objects.
	
	\begin{proposition}\label{prop:identify-derived}
		For every \(q\ge 0\),
		\[
		R^q\mathcal F(W)\cong \Ext_{B_E}^q(Q,W)
		\]
		as left \(G\)-modules. Consequently, the spectral sequence of
		Theorem~\ref{thm:spectral} takes the form
		\[
		E_2^{p,q}
		\cong
		H^p\!\bigl(G,\Ext_{B_E}^q(Q,W)\bigr)
		\Longrightarrow
		\Ext_{\Rep(D)}^{p+q}(Q,W).
		\]
	\end{proposition}
	
	\begin{proof}
		Let
		\[
		0\to W\to I^0\to I^1\to \cdots
		\]
		be an injective resolution in \(\Rep(D)\). By Lemma~\ref{lem:U-exact} and Proposition~\ref{prop:U-preserves-injectives},
		the complex \(U(I^\bullet)\) is an injective resolution of \(W\) in
		\(B_E\text{-}\mathsf{Mod}\). Hence the cohomology of the complex
		\[
		0\to \Hom_{B_E}(Q,U(I^0))\to \Hom_{B_E}(Q,U(I^1))\to \cdots
		\]
		computes \(\Ext_{B_E}^q(Q,W)\). Since \(U\) does not change the underlying
		\(B_E\)-module, this is exactly the complex
		\[
		0\to \Hom_{B_E}(Q,I^0)\to \Hom_{B_E}(Q,I^1)\to \cdots
		\]
		computing \(R^q\mathcal F(W)\). The \(G\)-action on the complex is the one from
		Definition~\ref{def:G-action-Hom}, so the resulting identification is
		\(G\)-equivariant.
	\end{proof}
	
	\subsection{Corollaries}
	
	We now extract the consequences most relevant for Maschke-type splitting.
	
	\begin{corollary}[Collapse in the Maschke regime]\label{cor:collapse}
		Assume that \(G\) is finite and \(\operatorname{char}(\kk)=0\). Then the spectral sequence collapses at the
		\(E_2\)-page, and for every \(n\ge 0\),
		\[
		\Ext_{\Rep(D)}^n(Q,W)\cong \Ext_{B_E}^n(Q,W)^G.
		\]
	\end{corollary}
	
\begin{proof}
	Since \(G\) is finite and \(\operatorname{char}(\kk)=0\), the group algebra \(\kk[G]\) is semisimple.
	Hence
	\[
	H^p(G,M)=0
	\qquad
	(p>0)
	\]
	for every left \(\kk G\)-module \(M\). In particular, Proposition~\ref{prop:F-acyclic-Maschke}
	applies. Moreover, by Proposition~\ref{prop:identify-derived}, the \(E_2\)-page of the spectral
	sequence is
	\[
	E_2^{p,q}\cong H^p\!\bigl(G,\Ext_{B_E}^q(Q,W)\bigr).
	\]
	
	Since \(H^p(G,-)=0\) for every \(p>0\), all terms with \(p>0\) vanish. Therefore the spectral
	sequence collapses at the \(E_2\)-page, and for every \(n\ge 0\),
	\[
	\Ext_{\Rep(D)}^n(Q,W)\cong E_2^{0,n}
	=
	H^0\!\bigl(G,\Ext_{B_E}^n(Q,W)\bigr)
	=
	\Ext_{B_E}^n(Q,W)^G.
	\]
\end{proof}
	
	\begin{corollary}[Splitting criterion]\label{cor:splitting-criterion}
		Assume that \(G\) is finite and \(\operatorname{char}(\kk)=0\). If
		\[
		\Ext_{B_E}^1(Q,W)^G=0,
		\]
		then every short exact sequence
		\[
		0\to W\to V\to Q\to 0
		\]
		in \(\Rep(D)\) splits.
	\end{corollary}
	
	\begin{proof}
		By Corollary~\ref{cor:collapse},
		\[
		\Ext_{\Rep(D)}^1(Q,W)\cong \Ext_{B_E}^1(Q,W)^G.
		\]
		Hence the hypothesis implies \(\Ext_{\Rep(D)}^1(Q,W)=0\). In any abelian category, this is
		equivalent to the splitting of every short exact sequence
		\[
		0\to W\to V\to Q\to 0.
		\]
	\end{proof}
	
	\begin{corollary}[Criterion for non-semisimplicity]\label{cor:nonsemisimple}
		Assume that \(G\) is finite and \(\operatorname{char}(\kk)=0\). If there exist objects \(Q,W\in\Rep(D)\) such that
		\[
		\Ext_{B_E}^1(Q,W)^G\neq 0,
		\]
		then the category \(\Rep(D)\) is not semisimple.
	\end{corollary}
	
	\begin{proof}
		Again by Corollary~\ref{cor:collapse},
		\[
		\Ext_{\Rep(D)}^1(Q,W)\cong \Ext_{B_E}^1(Q,W)^G.
		\]
		If the latter is nonzero, then \(\Ext_{\Rep(D)}^1(Q,W)\neq 0\). Therefore there exists a nontrivial
		extension of \(Q\) by \(W\) in \(\Rep(D)\), and such an extension cannot split. Hence
		\(\Rep(D)\) is not semisimple.
	\end{proof}
	
	\begin{remark}
		Theorem~\ref{thm:spectral} and its corollaries isolate the precise mechanism behind the failure of
		a full Maschke theorem for generalized digroups. The group side controls the invariant functor and
		therefore the \(p\)-direction of the spectral sequence, while the halo algebra controls the
		\(q\)-direction. The obstruction to semisimplicity is therefore genuinely two-layered.
	\end{remark}

\end{document}